\documentclass[twoside,12pt]{article}
\usepackage{amsmath,amsfonts,amscd,bezier}
\usepackage[usenames,dvipsnames]{pstricks}
\usepackage{epsfig}
\usepackage{tikz}
\usepackage{amsthm,amsmath,amsfonts,amssymb}
\usepackage{graphicx}
\usepackage{calrsfs}
\usepackage{hyperref}
\usepackage[latin1]{inputenc}             
\usepackage{color}

\numberwithin{equation}{section}

\newtheorem{theorem}{Theorem}[section]
\newtheorem{lemma}[theorem]{Lemma}
\newtheorem{corollary}[theorem]{Corollary}
\newtheorem{remark}[theorem]{Remark}
\newtheorem{definition}[theorem]{Definition}
\newtheorem{proposition}[theorem]{Proposition}

\begin{document}
\title{Dynamics of thermoelastic plate system with terms concentrated in the boundary: the lower semicontinuity of the global attractors}

\author{Gleiciane S. Arag\~ao\footnote{Partially supported by FAPESP 2019/04476-6, Brazil.} \\ {\footnotesize Departamento de Ci\^encias Exatas e da Terra, Universidade Federal de S\~ao Paulo,}\\ {\footnotesize 09913-030 Diadema SP, Brazil.}\\ {\footnotesize gleiciane.aragao@unifesp.br}\\ \\ Flank D. M. Bezerra\\ {\footnotesize Departamento de Matem\'atica, Universidade Federal da Para\'iba}\\ {\footnotesize 58051-900 Jo\~ao Pessoa PB, Brazil.}\\ {\footnotesize flank@mat.ufpb.br}\\ \\ Cl\'adio O. P. Da Silva\\ {\footnotesize Centro de Ci\^encias Humanas e Exatas, Universidade Estadual da Para\'iba}\\ {\footnotesize 58500-000 Monteiro PB, Brazil.}\\ {\footnotesize cladio@cche.uepb.edu.br}}

\maketitle

\begin{abstract}
In this paper we show the  lower semicontinuity of the global attractors of  autonomous thermoelastic plate systems with Neumann boundary conditions when some reaction terms are concentrated in a neighborhood of the boundary and this neighborhood  shrinks to boundary as a parameter $\varepsilon$ goes to zero.


\noindent \textit{keywords}: global attractor; thermoelastic plate systems; autonomous; concentrating terms; lower semicontinuity; dynamics.

\end{abstract}

\tableofcontents


\section{Introduction}

In this work we analyze the asymptotic  behavior of the global compact attractors of  autonomous thermoelastic plate systems with Neumann boundary conditions when some reaction terms are concentrated in a neighborhood of the boundary and this neighborhood  shrinks to boundary as a parameter $\varepsilon$ goes to zero. There has been numerous studies to investigate the dynamics, in the sense of  attractors, of systems when reaction terms are concentrated in a neighborhood of the boundary and this neighborhood  shrinks to boundary as a parameter $\varepsilon$ goes to zero, see for instance \cite{aragaobezerra0,aragaobezerra1,ABDS,aragaooliva1,aragaooliva2,GPMa0,GPMa,arrieta,anibal,AnibalAngela} and references therein. 

In this paper we continue the analysis made in \cite{ABDS}, and to better describe the problem we introduce some notations, let $\Omega$ be an open bounded smooth set in  $\mathbb{R}^{5}$ with  boundary $\Gamma=\partial \Omega$. We define the strip of width $\varepsilon$ and base $\partial \Omega$ as
$$
\omega_{\varepsilon}=\{x-\sigma \stackrel{\rightarrow}{n}(x): \ \mbox{$x\in \Gamma$  \ and  \ $\sigma \in [0,\varepsilon)$}\}, 
$$
for sufficiently small $\varepsilon$, say $0< \varepsilon \leqslant \varepsilon_{0}$, where $\stackrel{\rightarrow}{n}(x)$ denotes the outward normal vector at $x\in \Gamma$. We note that the set $\omega_{\varepsilon}$ has Lebesgue measure $\left|\omega_{\varepsilon}\right|=O(\varepsilon)$ with $\left|\omega_{\varepsilon}\right|\leqslant k\left|\Gamma\right|\varepsilon$, for some $k> 0$ independent of $\varepsilon$, and that for small $\varepsilon$, the set $\omega_{\varepsilon}$ is a neighborhood of $\Gamma$ in $\overline{\Omega}$, that collapses to the boundary when the parameter $\varepsilon$ goes to zero, see Figure \ref{figomega}. 
\begin{figure}[!h]\label{figomega}
\begin{center}
    \includegraphics[scale=.4]{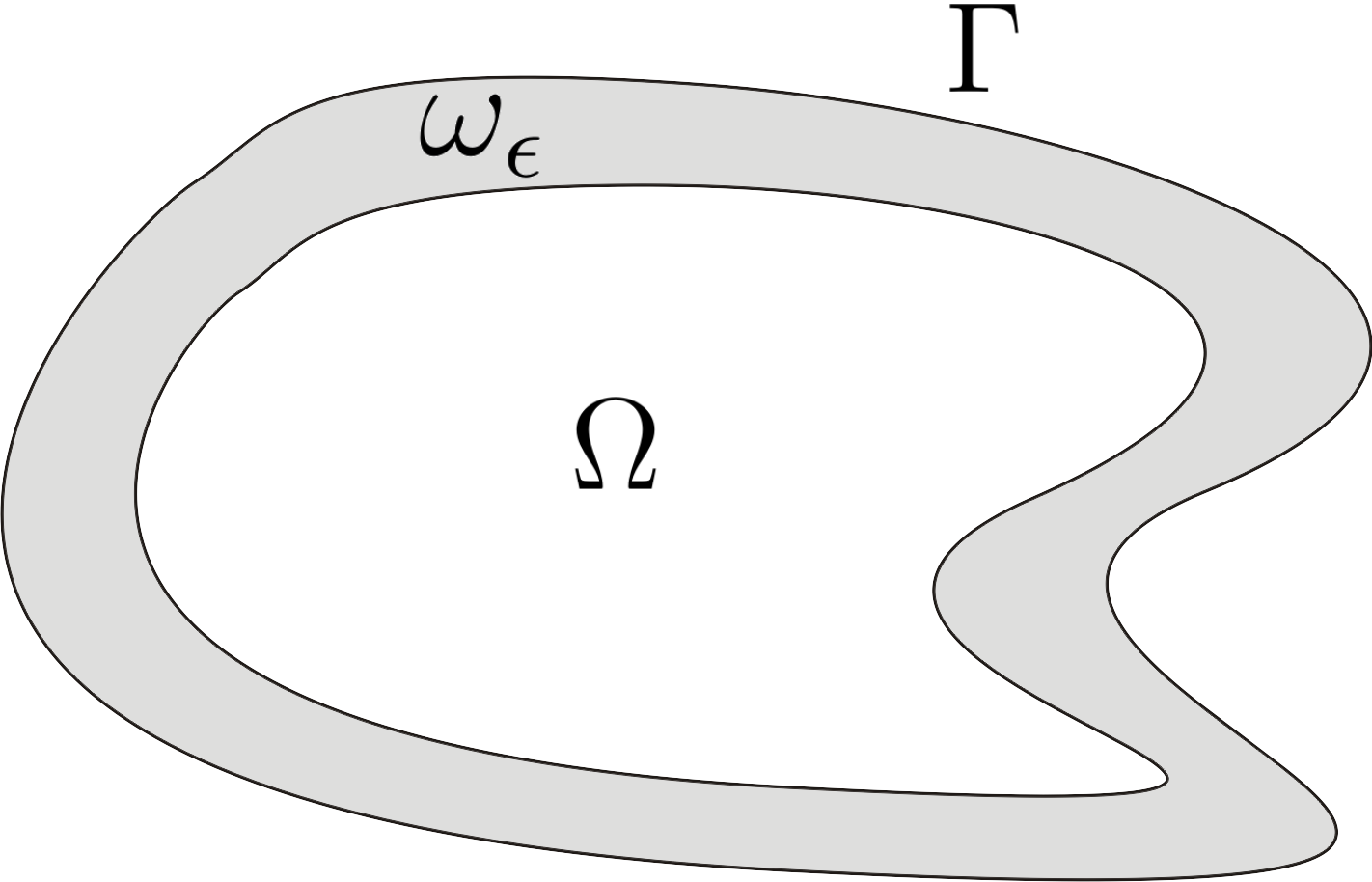}
\end{center}
\caption{The set $\omega_{\epsilon}\subset\overline{\Omega}$. This figure has been extracted of \cite{ABDS}.} 
\end{figure}

In \cite{ABDS} we show the existence, uniform boundedness and upper semicontinuity of the global attractors at $\varepsilon=0$ of the autonomous thermoelastic plate system
\begin{equation}\label{PPrin_1}
\begin{cases}
\partial_t^2u^\varepsilon+\Delta^2 u^\varepsilon+ u^\varepsilon +\Delta \theta^\varepsilon- \theta^\varepsilon=f(u^\varepsilon)+\dfrac{1}{\varepsilon}\chi_{\omega_{\varepsilon}}g(u^\varepsilon)& \mbox{in } \Omega\times(0,+\infty),\\
\partial_t\theta^\varepsilon-\Delta \theta^\varepsilon+ \theta^\varepsilon -\Delta \partial_tu^\varepsilon+ \partial_tu^\varepsilon=0& \mbox{in } \Omega\times(0,+\infty),\\
\dfrac{\partial u^\varepsilon}{\partial\vec{n}}=0,\ \ \dfrac{\partial (\Delta u^\varepsilon)}{\partial\vec{n}}=0,\ \ \dfrac{\partial \theta^\varepsilon}{\partial\vec{n}}=0 & \mbox{on}\ \Gamma\times(0,+\infty),\\
u^\varepsilon(0)=u_0\in H^2(\Omega),\ \ u_t^\varepsilon(0)=v_0\in L^2(\Omega),\ \ \theta^\varepsilon(0)=\theta_0\in L^2(\Omega),
\end{cases}
\end{equation}
where $\chi_{\omega_{\varepsilon}}$ denotes the characteristic function of the set $\omega_{\varepsilon}$. As in \eqref{PPrin_1} the nonlinear term $g(u^\varepsilon)$ is only effective on the region $\omega_{\varepsilon}$ which collapses to $\Gamma$ as $\varepsilon\to0$, then it is reasonable to expect that the family of solutions $u^{\varepsilon}$ of \eqref{PPrin_1}  will converge to a solution of an equation of the same type with nonlinear boundary condition on $\Gamma$. Indeed, we show that the ``limit problem'' for the autonomous thermoelastic plate system \eqref{PPrin_1} is given by
\begin{equation}
\label{PPrin_2}
\begin{cases}
\partial_t^2u+\Delta^2 u+ u +\Delta \theta- \theta=f(u) & \mbox{in}\ \Omega\times(0,+\infty),\\
\partial_t\theta-\Delta \theta+ \theta -\Delta \partial_tu+ \partial_tu=0& \mbox{in}\ \Omega\times(0,+\infty),\\
\dfrac{\partial u}{\partial\vec{n}}=0,\ \ \dfrac{\partial (\Delta u)}{\partial\vec{n}}=-g(u),\ \ \dfrac{\partial \theta}{\partial\vec{n}}=0 & \mbox{on}\ \Gamma\times(0,+\infty),\\
u(0)=u_0\in H^2(\Omega),\ \ u_t(0)=v_0\in L^2(\Omega),\ \ \theta(0)=\theta_0\in  L^2(\Omega).
\end{cases}
\end{equation}

We  consider $j:\mathbb{R}\to\mathbb{R}$ a $\mathcal{C}^2-$function and assume that it satisfies the growth estimates
\begin{equation}
\label{Dcrescimento}
|j(s)|+|j'(s)|+|j''(s)|\leqslant K,\quad \forall  s \in\mathbb{R},
\end{equation}
for some constant $K>0$, we also assume the standard dissipative assumption given by

\begin{equation}
\label{DissipCond}
\limsup_{|s|\to+\infty}\dfrac{j(s)}{s}\leqslant0,
\end{equation}
with  $j=f$ or $j=g$. We note that \eqref{DissipCond} is equivalent to saying that for any $\gamma >0$ there exists $c_{\gamma}>0$ such that
\begin{equation}
\label{EqDissipCond}
sj(s)\leqslant \gamma s^2+c_{\gamma}, \quad \forall s\in \mathbb{R}.
\end{equation}

Here,  we will prove the lower semicontinuity of the global attractors at $\varepsilon=0$ of the problems \eqref{PPrin_1} and \eqref{PPrin_2}, but for this end we need to show a result of continuity of equilibrium solutions of \eqref{PPrin_1} and \eqref{PPrin_2}, that is, the solutions of the  eliptic problems associated to \eqref{PPrin_1} and \eqref{PPrin_2} and we also show the continuity of local unstable manifold around of the set of equilibria.

To study the continuity of the set of equilibria we need to show the upper and lower semicontinuity. The upper semicontinuity is a direct consequence of the upper semicontinuity of the global attractors. To the lower semicontinuity we will assume that all equilibrium solutions of the problem \eqref{PPrin_2} are hyperbolic. We will show that the set of equilibria of \eqref{PPrin_2}, which we will denote by $\mathcal{E}_0$ has cardinality $k$, with elements different $w^*_1,\dots,w^*_k$. After we will show that there exist $\varepsilon_0$ such that the problem \eqref{PPrin_1}, has exactly $k$ equilibrium solutions, which we will denote by $w^*_{\varepsilon,1},\dots,w^*_{\varepsilon,k}$, for $\varepsilon \in (0,\varepsilon_0]$. Moreover, we will obtain the convergence $w^*_{\varepsilon,i} \to w^*_{i}$ as $\varepsilon \to 0$, for $i=1,\ldots,k$. 

To show the continuity of local unstable manifold around of the equilibrium, we first linearize the abstract problem around of the equilibrium solution $w^*_{\varepsilon,i}$, then we show the existence of this manifold as the graph of a map Lipschitz, and using the continuity of linearized semigroups we will show that the local unstable manifold, which we will denote by $W^{u}(w^*_{\varepsilon,i})$ are continuous at $\varepsilon=0$.      

With these two results and verifying that \eqref{PPrin_1} and \eqref{PPrin_2} have gradient structure we conclude the lower semicontinuity of the family of global attractors.

This paper is organized as follows. In Section \ref{well-posedness}, we will present some notations and we will define the abstract problems associated to the initial-boundary value problems \eqref{PPrin_1} and \eqref{PPrin_2}. Also we will present a result that ensure us the sectoriality of operator, concluding thus that there is an analytic semigroup generated by our operator. After we will see properties of the nonlinearities and of your derivatives. The Section \ref{Sec:ExAtt_UpperSem} is dedicated to the results on existence, characterization and uniform bounds of the global attractor, as well as the convergence of the nonlinear semigroups associated to the abstract problems, that was used to prove the upper semicontinuity of global attractors at $\varepsilon=0$, we refer to our results in \cite{ABDS}. In Section \ref{Sec:Cont_Eq} we will study the time independent solutions, that is, the equilibrium solutions of the problems \eqref{PPrin_1} and \eqref{PPrin_2}. Specifically we will prove the continuity of the set of equilibria. Finally, in Section \ref{Sec:Lower_At} we will prove the continuity of local unstable manifold around of the equilibrium and the lower semicontinuity of the global attractors of the problems \eqref{PPrin_1} and \eqref{PPrin_2} at $\varepsilon=0$. 


\section{Abstract setting}
\label{well-posedness}

To better explain the results in the paper, initially, we will define the abstract problems associated to \eqref{PPrin_1} and \eqref{PPrin_2}. After we will see properties of the nonlinearities and of your derivatives.


\subsection{Functional spaces}

Let us consider the  Hilbert space $Y:=L^2(\Omega)$ and the unbounded linear operator $\varLambda:D(\varLambda)\subset Y\to Y$ defined by 
\[
\varLambda u=(-\Delta)^2 u, \quad u\in D(\varLambda),
\]
with domain
\[
D(\varLambda):=\Big\{u\in H^4(\Omega) \;:\; \dfrac{\partial u}{\partial\vec{n}}=\dfrac{\partial (\Delta u)}{\partial\vec{n}}=0\ \mbox{on}\ \Gamma\Big\}.
\]

The operator  $\varLambda$ has a discrete spectrum formed of eigenvalues satisfying
\[
0=\mu_1\leqslant\mu_2\leqslant\cdots\leqslant\mu_n\leqslant\cdots,\quad\lim_{n\to\infty}\mu_n=\infty.
\]

Since this operator turns out to be sectorial in $Y$ in the sense of Henry \cite[Definition 1.3.1]{He} and Cholewa and D\l otko \cite[Example 1.3.9]{JanTomasz}, associated to it there is a scale of Banach spaces $Y^\alpha$, $\alpha \in \mathbb{R}$, denoting the domain of the fractional power operators associated with $\varLambda$, that is, $Y^\alpha:=D(\varLambda^\alpha), \alpha \geqslant 0$. Let us consider $Y^\alpha$ endowed with the  norm $\|(\cdot)\|_{Y^\alpha}=\|\varLambda^\alpha (\cdot)\|_Y+\|(\cdot)\|_Y, \alpha \geqslant 0$. The fractional power spaces are related to the Bessel Potentials spaces $H^{s}(\Omega)$, $s\in \mathbb{R}$,  and
it is well know that 
\[
Y^\alpha \hookrightarrow H^{2\alpha}(\Omega),\quad Y^{-\alpha}=(Y^\alpha)',\ \ \alpha\geqslant0,
\]
with 
\[
Y^{\frac{1}{2}}=\Big\{u\in H^2(\Omega):\ \dfrac{\partial u}{\partial\vec{n}}=0\ \mbox{on}\ \Gamma\Big\}.
\]

We also have 
\[
Y^{-\frac{1}{2}}=(Y^{\frac{1}{2}})',\quad Y=Y^0=L^2(\Omega)\quad \mbox{and}\quad Y^1=D(\Lambda).
\]

Since the problem \eqref{PPrin_2} has a nonlinear term on boundary, choosing $\frac{1}{2} < s \leqslant 1$ and using the standard trace theory results that for any function $v \in H^{s}(\Omega)$, the trace of $v$ is well defined and lies in $L^{2}(\Gamma)$. Moreover, the scale of negative exponents $Y^{-\alpha}$, for $\alpha > 0$, is necessary to introduce the nonlinear term of \eqref{PPrin_2} in the abstract equation, since we are using the operator  $\varLambda$ with homogeneous boundary conditions. If we consider the realizations of $\varLambda$ in this scale, then the operator $\varLambda_{-\frac{1}{2}} \in \mathcal{L}(Y^{\frac{1}{2}},Y^{-\frac{1}{2}})$ is given by 
$$
\langle \varLambda_{-\frac{1}{2}}u, v\rangle_Y = \int_\Omega \Delta u\Delta vdx+\int_\Omega uv dx, \quad u,v \in Y^{\frac{1}{2}}.
$$
With some abuse of notation we will identify all different realizations of this operator and we will write them all as $\Lambda$.

We also consider the operator $\varLambda+I:D(\varLambda+I)\subset Y\to Y$, it is a  positive defined and sectorial operator in $Y$ in the sense of Henry \cite[Definition 1.3.1]{He} and Cholewa and D\l otko \cite[Example 1.3.9]{JanTomasz}, associated to it there is a scale of Banach spaces (which are fractional power spaces)  $D((\varLambda+I)^\alpha)$, $ \alpha \geqslant 0$, domain of the operator $(\varLambda+I)^\alpha$. Let us consider $D((\varLambda+I)^\alpha)$ endowed with the graph norm $\|(\cdot)\|_{D((\varLambda+I)^\alpha)}=\|(\varLambda+I)^\alpha(\cdot)\|_Y, \alpha \geqslant 0$ ($0\in\rho((\varLambda+I)^\alpha)$). Consequentely, by Cholewa and D\l otko \cite[Corollary 1.3.5]{JanTomasz} and $D(\varLambda+I)=D(\varLambda)$, we also have that
\[
Y^\alpha=[Y,D(\varLambda)]_{\alpha}=[Y,D(\varLambda+I)]_{\alpha} =D((\varLambda+I)^\alpha), \quad 0\leqslant\alpha\leqslant1,
\]
endowed with equivalent norms. 

The operator  $\varLambda+I$ has a discrete spectrum formed of eigenvalues satisfying
\[
1=\mu^I_1\leqslant\mu^I_2\leqslant\cdots\leqslant\mu^I_n\leqslant\cdots,\quad\lim_{n\to\infty}\mu^I_n=\infty.
\]

Also, let us consider the following Hilbert spaces
\[
X=X^0=Y^{\frac{1}{2}}\times Y \times Y
\]
equipped with the inner product
\[
\Big\langle\Big(\begin{smallmatrix}u_1\\ v_1\\  \theta_1\end{smallmatrix}\Big),\Big(\begin{smallmatrix}u_2\\ v_2\\ \theta_2\end{smallmatrix}\Big)\Big\rangle_X =\langle u_1,u_2\rangle_{Y^{\frac{1}{2}}}+\langle v_1,v_2 \rangle_Y + \langle \theta_1,\theta_2 \rangle_{Y},
\]
where $\langle \cdot,\cdot \rangle_Y$ is the usual inner product in $L^{2}(\Omega)$, and
\[
\mathcal{H}= H^{2}(\Omega) \times H^{-s}(\Omega) \times L^{2}(\Omega)
\]
equipped with the usual inner product with $\frac{1}{2} < s \leqslant 1$.

We define the unbounded linear operator $\mathbb{A}:D(\mathbb{A})\subset X\to X$ by
\begin{equation}\label{Def4a}
\mathbb{A}\Big(\begin{smallmatrix} u\\ v\\ \theta \end{smallmatrix}\Big)=\begin{pmatrix}0 & I & 0 \\ -\varLambda-I & 0 & \varLambda^{\frac{1}{2}}+I\\ 0 & -\varLambda^{\frac{1}{2}}-I& -\varLambda^{\frac{1}{2}}-I\end{pmatrix} \Big(\begin{smallmatrix} u\\ v\\ \theta \end{smallmatrix}\Big)=\begin{pmatrix} v \\ -\varLambda u-u+\varLambda^{\frac{1}{2}}\theta+\theta \\ -\varLambda^{\frac{1}{2}}v-v - \varLambda^{\frac{1}{2}}\theta-\theta \end{pmatrix}, \quad \Big(\begin{smallmatrix} u\\ v\\ \theta \end{smallmatrix}\Big) \in D(\mathbb{A}), 
\end{equation}
with domain
\begin{equation}\label{Def3a}
D(\mathbb{A})=Y^1\times Y^{\frac{1}{2}} \times Y^{\frac{1}{2}}.
\end{equation}


For each $\varepsilon \in (0,\varepsilon_0]$, we write \eqref{PPrin_1} in the abstract form as
\begin{equation}
\label{AP1}
\begin{cases}
\displaystyle \dfrac{dw^{\varepsilon}}{dt}=\mathbb{A} w^{\varepsilon}+F_\varepsilon(w^{\varepsilon}), \quad  t> 0,\\
w^{\varepsilon}(0)=w_0,
\end{cases}
\end{equation}
with $\partial_tu^{\varepsilon}=v^{\varepsilon}$,
\[
w^{\varepsilon}=\Big(\begin{smallmatrix} u^{\varepsilon}\\ v^{\varepsilon}\\ \theta^{\varepsilon}\end{smallmatrix}\Big) , \quad w_{0}=\Big(\begin{smallmatrix}u_0\\ v_0\\ \theta_0\end{smallmatrix}\Big) \in X
\]
and nonlinear map $F_\varepsilon: X  \to  \mathcal{H}$, with $\frac{1}{2} < s \leqslant 1$, defined by
\begin{equation*}
F_\varepsilon (w)=\begin{pmatrix}0\\ f_{\Omega}(u)+\dfrac{1}{\varepsilon}\chi_{\omega_{\varepsilon}}g_{\Omega}(u)\\ 0\end{pmatrix}, \quad  w=\Big(\begin{smallmatrix}u\\ v\\ \theta\end{smallmatrix}\Big) \in X,
\end{equation*}
where $f_{\Omega},\dfrac{1}{\varepsilon}\chi_{\omega_{\varepsilon}}g_{\Omega} :H^2(\Omega)\to H^{-s}(\Omega)$ are the operators, respectively, given by 
\begin{equation}
\label{finterior}
\langle  f_{\Omega}(u),\varphi \rangle =\int_{\Omega} f(u) \varphi dx,  \quad \mbox{$ u\in H^2(\Omega)$  and $ \varphi \in H^{s}(\Omega)$} 
\end{equation}
and
\begin{equation}
\label{ginterior}
\Big\langle  \dfrac{1}{\varepsilon}\chi_{\omega_{\varepsilon}}g_{\Omega}(u),\varphi \Big\rangle =\frac{1}{\varepsilon} \int_{\omega_{\varepsilon}} g(u) \varphi dx,
 \quad \mbox{$ u\in H^{2}(\Omega)$  and $ \varphi \in H^{s}(\Omega)$}.
\end{equation}

While the problem \eqref{PPrin_2} can be written in the abstract form as
\begin{equation} 
\label{AP1b}
\begin{cases}
\displaystyle \dfrac{dw}{dt}=\mathbb{A} w+F_0(w),\quad t> 0,\\
w(0)=w_0,
\end{cases}
\end{equation}
with $\partial_tu=v$,
\[
w=\Big(\begin{smallmatrix}u\\ v\\ \theta\end{smallmatrix}\Big)
\]
and nonlinear map $F_0: X \to  \mathcal{H}$, with $\frac{1}{2} < s \leqslant 1$,  defined by
\begin{equation*}
F_0(w) =\begin{pmatrix}0\\ f_{\Omega}(u)+g_{\Gamma}(u)\\ 0\end{pmatrix}, \quad  w=\Big(\begin{smallmatrix}u\\ v\\ \theta\end{smallmatrix}\Big)\in  X,
\end{equation*}
where   $f_{\Omega}$ is defined in (\ref{finterior}) and $g_{\Gamma} :H^2(\Omega)\to H^{-s}(\Omega)$ is the operator given by 
\begin{equation}
\label{gfronteira}
\langle  g_{\Gamma}(u),\varphi \rangle =\int_{\Gamma} \gamma(g(u)) \gamma(\varphi) dS, \quad \mbox{$ u\in H^{2}(\Omega)$  and $ \varphi \in H^{s}(\Omega)$},   
\end{equation}
where $\gamma: H^{s}(\Omega)\to  L^{2}(\Gamma)$ is the trace operator, to according with Triebel \cite{T}.

\begin{theorem}\label{T24}
The unbounded linear operator $-\mathbb{A}$ such that $\mathbb{A}:D(\mathbb{A})\subset X\to X$ is defined in \eqref{Def4a}-\eqref{Def3a} is sectorial.
\end{theorem}
\begin{proof}
For the proof see \cite[Theorem 3]{ABDS}.	
\end{proof}	
\begin{remark} The following startments are hold.
\begin{itemize}
\item[(i)] Zero is in the resolvent set of $\mathbb{A}$ and 
\[
\mathbb{A}^{-1}=\begin{pmatrix} -(\varLambda+I)^{-1}(\varLambda^{\frac{1}{2}}+I) & - (\varLambda+I)^{-1} & - (\varLambda+I)^{-1} \\ I & 0 & 0\\ -I & 0 & -(\varLambda^{\frac{1}{2}}+I)^{-1} \end{pmatrix}. 
\] 
\item[(ii)]
Denote by $X_{-1}$ the extrapolation space of $X=Y^{\frac{1}{2}}\times Y \times Y$ generated by the operator  $\mathbb{A}^{-1}$. The following equality holds
\[
X_{-1}= Y\times Y^{-\frac{1}{2}} \times Y^{-\frac{1}{2}}.
\] 
In fact, recall first that $X_{-1}$ is the completion of the normed space $(X,\|\mathbb{A}^{-1}\cdot\|)$. Note that
\[
\left\|\mathbb{A}^{-1} \Big(\begin{smallmatrix} u\\ v\\ \theta \end{smallmatrix}\Big)\right\|_X\leqslant C_1\left\|\Big(\begin{smallmatrix} u\\ v\\ \theta \end{smallmatrix}\Big)\right\|_{X_{-1}},
\]
for some constant $C_1>0$. Well as we have
\[
\begin{split}
\left\| \Big(\begin{smallmatrix} u\\ v \\ \theta \end{smallmatrix}\Big)\right\|_{X_{-1}}\leqslant C_2\left\|\mathbb{A}^{-1}\Big(\begin{smallmatrix} u\\ v \\ \theta \end{smallmatrix}\Big)\right\|_{X},
\end{split}
\]
for some constant $C_2>0$.

So we conclude that the completion of $(X,\|\mathbb{A}^{-1}\cdot\|_{X})$ and $(X,\|\cdot\|_{X_{-1}})$ coincide.
\end{itemize}
\end{remark}

Note that the operator $\mathbb{A}$ can be  extended to its closed $X_{-1}-$realization, see Amann \cite{A}, which we will still denote by the same symbol  so that $\mathbb{A}$ considered in $X_{-1}$ is then sectorial positive operator. Our next concern will be to obtain embedding of the spaces from the fractional powers scale $X_{\alpha-1}$, $\alpha\geqslant 0$, generated by $(\mathbb{A}, X_{-1})$.

\begin{remark}
Below we have a partial description of the fractional power spaces scale for $\mathbb{A}$: for convenience we denote $X$ by $X_0$, then
\[
X_0 \hookrightarrow X_{\alpha-1} \hookrightarrow X_{-1},\quad \mbox{for all}\ 0<\alpha<1,
\]
where
\[
X_{\alpha-1} = [X_{-1},X_0]_{\alpha} = Y^{\frac{\alpha}{2}} \times Y^{\frac{\alpha-1}{2}} \times Y^{\frac{\alpha-1}{2}},
\]
where $[\cdot,\cdot]_{\alpha}$ denotes the complex interpolation functor (see Triebel \cite{T}). The first equality follows from Theorem \ref{T24} (since $0\in \rho(\mathbb{A})$) see Amann \cite[Example 4.7.3 (b)]{A} and the second equality follows from Carvalho and Cholewa \cite[Proposition 2]{CC}.
\end{remark}


\subsection{Nonlinearities}

The behavior of the nonlinearity $F_{\varepsilon}$ was studied in \cite{ABDS}. The main results  are given below.

\begin{lemma}
\label{Lipschitz} 
Suppose that $f$ and $g$ satisfy the growth estimate (\ref{Dcrescimento}) and $\frac{1}{2}<s\leqslant1$. Then:
\begin{itemize}
\item[(i)] There exists $C> 0$, independent of $\varepsilon$, such that 
\begin{equation}
\label{equ0}
\left\| F_{\varepsilon}(w) \right\|_{\mathcal{H}}\leqslant C,  \quad \mbox{$w  \in X$ and $\varepsilon \in [0,\varepsilon_0]$}.
\end{equation}
\item[(ii)] For each $\varepsilon \in [0,\varepsilon_0]$, the map $F_{\varepsilon}: X \to  \mathcal{H}$ is globally Lipschitz, uniformly in $\varepsilon$. 
\item[(iii)] For each $w \in X$, we have
$$
\left\| F_{\varepsilon}(w)- F_{0}(w) \right\|_{ \mathcal{H}} \to 0, \quad \mbox{as $\varepsilon \to 0$}.
$$
Furthermore, this limit is uniform for $w\in X$ such that $\left\| w \right\|_{X}\leqslant R$, for some $R> 0$.
\item[(iv)] If $w_{\varepsilon} \to w$ in $X$, as $\varepsilon \to 0,$ then
$$
\left\| F_{\varepsilon}(w_{\varepsilon})- F_{0}(w) \right\|_{ \mathcal{H}} \to 0, \quad \mbox{as $\varepsilon \to 0$}.
$$
\end{itemize}
\end{lemma}

\begin{proof}
For the proof see \cite[Lemma 3]{ABDS}.
\end{proof}

From Lemma \ref{Lipschitz} follows that the map $F_{\varepsilon}: X \to \mathcal{H}$ is bounded, uniformly in $\varepsilon$, in bounded set of $X$, and it is locally Lipschitz, uniformly in $\varepsilon$. Thus, it follows from \cite[Theorem 4.2.1]{Ha} that given $w_0 \in X$, there is an unique local solution $w^{\varepsilon}(t,w_0)$ of  (\ref{AP1}), with $\varepsilon \in (0,\varepsilon_0]$, defined on a maximal interval of existence  $[0,t^{\varepsilon}_{max}(w_0))$, and there is an unique local solution $w(t,w_0)$ of  (\ref{AP1b}) defined on a maximal interval of existence $[0,t_{max}(w_0))$. Moreover, these solutions depend continuously on the initial data. 



We define the maps $Df_{\Omega},\frac{1}{\varepsilon}\chi_{\omega_{\varepsilon}}Dg_{\Omega},Dg_{\Gamma}:H^2(\Omega)\to \mathcal{L}(H^2(\Omega),H^{-s}(\Omega))$, with $\frac{1}{2}<s\leqslant 1$, respectively by
\begin{equation}
\label{Dfinterior}
\langle  Df_{\Omega}(u)\cdot h,\varphi \rangle =\int_{\Omega} f'(u)h \varphi dx,  \quad \mbox{$u,h\in H^{2}(\Omega)$  and $ \varphi \in H^{s}(\Omega)$}, 
\end{equation}
\begin{equation}
\label{Dginterior}
\langle  \dfrac{1}{\varepsilon}\chi_{\omega_{\varepsilon}}Dg_{\Omega}(u)\cdot h,\varphi \rangle =\frac{1}{\varepsilon} \int_{\omega_{\varepsilon}} g'(u)h \varphi dx,
 \quad \mbox{$ u,h\in H^{2}(\Omega)$  and $ \varphi \in H^{s}(\Omega)$}
\end{equation}
and
\begin{equation}
\label{Dgfronteira}
\langle  Dg_{\Gamma}(u)\cdot h,\varphi \rangle =\int_{\Gamma} \gamma(g'(u)h) \gamma(\varphi) dS, \quad \mbox{$ u,h\in H^{2}(\Omega)$  and $ \varphi \in H^{s}(\Omega)$},   
\end{equation}
where $\gamma: H^{s}(\Omega)\to  L^{2}(\Gamma)$ is the trace operator.

\begin{lemma}
\label{Lipschitz_01} 
Suppose that $f$ and $g$ satisfy the growth estimates (\ref{Dcrescimento}) and $\frac{1}{2}<s\leqslant1$. Then:
\begin{itemize}
\item[(i)] $f_{\Omega},\frac{1}{\varepsilon}\chi_{\omega_{\varepsilon}}g_{\Omega},g_{\Gamma}:H^{2}(\Omega) \to H^{-s}(\Omega)$ are  Fr\'echet differentiable, uniformly in $\varepsilon$, and your Fr\'echet differentials are respectively given by (\ref{Dfinterior}), (\ref{Dginterior}) and (\ref{Dgfronteira}). Consequently, for each $\varepsilon\in[0,\varepsilon_0]$, $F_\varepsilon:X\to \mathcal{H}$ is also Fr\'echet differentiable, uniformly in $\varepsilon$;
\item[(ii)] $Df_{\Omega},\frac{1}{\varepsilon}\chi_{\omega_{\varepsilon}}Dg_{\Omega},Dg_{\Gamma}:H^2(\Omega)\to \mathcal{L}(H^2(\Omega),H^{-s}(\Omega))$ are globally Lipschitz, uniformly in $\varepsilon$. Consequently, for each $\varepsilon\in[0,\varepsilon_0]$, $DF_\varepsilon: X\to \mathcal{L}(X,\mathcal{H})$ is also globally Lipschitz, uniformly in $\varepsilon$.
\end{itemize}
\end{lemma}

\begin{proof}
For the proof of item $(i)$ see \cite[Lemma 4]{ABDS}, and  item $(ii)$ see \cite[Lemma 5]{ABDS}.
\end{proof}

Under the assumptions of Lemma \ref{Lipschitz_01}, we have that the map $F_\varepsilon$ is continuously Fr\'echet differentiable. Now, it follows from \cite[Theorem 4.2.1]{Ha} that the solutions of \eqref{AP1} and \eqref{AP1b} are continuously differentiable with respect to initial conditions.

Now, we prove a result of uniform boundedness and convergence of the Fr\'echet differential of the nonlinearity $F_{\varepsilon}$.

\begin{lemma} \label{resultsconvnonlinearity}
Suppose that $f$ and $g$ satisfy the growth estimates (\ref{Dcrescimento}) and $\frac{1}{2}<s\leqslant 1$. Then:
\begin{enumerate}
\item[(i)] There exists $k>0$, independent of $\varepsilon$, such that
\[
\|DF_{\varepsilon}(w)\|_{\mathcal{L}(X,\mathcal{H})}\leqslant k, \quad \mbox{$w \in X$ and $\varepsilon \in [0, \varepsilon_{0}]$}. 
\]

\item[(ii)] For each $w\in X$, we have
\[
\|DF_{\varepsilon}(w)-DF_{0}(w)\|_{\mathcal{L}(X,\mathcal{H})}\to 0, \quad \mbox{as $\varepsilon\to0$},
\]
and this limit is uniform for $w\in X$ such that $\|w\|_{X}\leqslant R$, for some $R>0$.

\item[(iii)] If $w_{\varepsilon}\to w$ in $X$, as $\varepsilon\to 0$, then
\[
\|DF_{\varepsilon}(w_{\varepsilon})-DF_{0}(w)\|_{\mathcal{L}(X,\mathcal{H})}\to 0, \quad \mbox{as $\varepsilon\to0$}.
\]

\item[(iv)] If $w_{\varepsilon}\to w$ in $X$, as $\varepsilon\to 0$, and ${\bf h}_\varepsilon\to {\bf h}$ in $X$, as $\varepsilon\to 0$, then
\[
\|DF_{\varepsilon}(w_{\varepsilon}) {\bf h}_{\varepsilon}-DF_{0}(w) {\bf h}\|_{\mathcal{H}}\to 0, \quad \mbox{as $\varepsilon\to0$}.
\]
\end{enumerate}
\end{lemma}

\begin{proof}
\noindent {\it (i)} Let $w=\Big(\begin{smallmatrix}u\\ v\\ \theta\end{smallmatrix}\Big) \in X$ and $\varepsilon \in [0, \varepsilon_{0}]$, we have
$$
\|DF_{\varepsilon}(w)\|_{\mathcal{L}(X,\mathcal{H})}=\displaystyle \sup_{\begin{array}{c}
{\bf h} \in X\\
\left\|{\bf h}\right\|_{X}=1
\end{array}} 
\| DF_{\varepsilon}(w) {\bf h} \|_{\mathcal{H} }.
$$
Note that,  for each ${\bf h} = \Big(\begin{smallmatrix} h_1\\ h_2\\ h_3\end{smallmatrix}\Big) \in X$,
$$
\begin{array}{lll}
\| DF_{\varepsilon}(w) {\bf h} \|_{\mathcal{H}}=\left\|Df_{\Omega}(u)\cdot h_1+\dfrac{1}{\varepsilon}\chi_{\omega_{\varepsilon}}Dg_{\Omega}(u)\cdot h_1\right\|_{H^{-s}(\Omega)}, \quad \varepsilon \in (0,\varepsilon_0], \\
\\
\| DF_{0}(w) {\bf h} \|_{\mathcal{H}}=\|Df_{\Omega}(u)\cdot h_1+Dg_{\Gamma}(u)\cdot h_1\|_{H^{-s}(\Omega)},
\end{array}
$$
where the maps $Df_{\Omega},\frac{1}{\varepsilon}\chi_{\omega_{\varepsilon}}Dg_{\Omega}$ and $Dg_{\Gamma}$ are given respectively by \eqref{Dfinterior}, \eqref{Dginterior} and \eqref{Dgfronteira}.

Similarly to \cite[Lemma 4]{ABDS}, we have that there exist $k_1, k_2,k_3>0$ independents of $\varepsilon$ such that

\begin{equation}\label{F1}
\|Df_{\Omega}(u)\cdot h_1\|_{H^{-s}(\Omega)}\leqslant k_1\|h_1\|_{H^2(\Omega)}, \quad  h_1 \in H^2(\Omega),
\end{equation}

\begin{equation}\label{F2}
\Big \|\dfrac{1}{\varepsilon}\chi_{\omega_{\varepsilon}}Dg_{\Omega}(u)\cdot h_1 \Big \|_{H^{-s}(\Omega)}\leqslant k_2 \|h_1\|_{H^2(\Omega)} , \quad  h_1 \in H^2(\Omega),
\end{equation}

\begin{equation}\label{F3}
\|Dg_{\Gamma}(u)\cdot h_1\|_{H^{-s}(\Omega)}\leqslant k_3 \|h_1\|_{H^2(\Omega)}  , \quad  h_1 \in H^2(\Omega).
\end{equation}
Therefore, the result follows from \eqref{F1}, \eqref{F2} and \eqref{F3}.

\vspace{0.2cm}
\noindent {\it (ii)} For each $w\in X$, notice that
$$
\|DF_{\varepsilon}(w)-DF_{0}(w)\|_{\mathcal{L}(X,\mathcal{H})}= \left\|\frac{1}{\varepsilon}\chi_{\omega_{\varepsilon}}Dg_{\Omega}(u)-Dg_{\Gamma}(u)\right\|_{\mathcal{L}(H^{2}(\Omega),H^{-s}(\Omega))}.
$$

As in \cite[Lemma 5.2]{AnibalAngela} we can prove that there exists $M(\varepsilon,R)$ with $M(\varepsilon,R)\to 0$ as $\varepsilon \to 0$ such that
$$
\begin{array}{lll}
\displaystyle \left| \langle \dfrac{1}{\varepsilon}\chi_{\omega_{\varepsilon}}Dg_{\Omega}(u)\cdot h_1- Dg_{\Gamma}(u)\cdot h_1,\varphi\rangle  \right|  = \displaystyle \left| \frac{1}{\varepsilon}\int_{\omega_{\varepsilon}}g'(u)h_1\varphi dx -\int_{\Gamma}\gamma(g'(u)h_1)\gamma(\varphi) dS \right| \\

\leqslant  \displaystyle M(\varepsilon,R) \left\| h_1 \right\|_{H^{2}(\Omega)} \left\| \varphi \right\|_{H^{1}(\Omega)},\quad \mbox{$\forall h_1\in H^{2}(\Omega)$ and $\forall \varphi \in H^{1}(\Omega)$}.
\end{array}
$$
Thus, 
\begin{equation}
\label{F4}
\left\| \dfrac{1}{\varepsilon}\chi_{\omega_{\varepsilon}}Dg_{\Omega}(u)- Dg_{\Gamma}(u)\right\|_{\mathcal{L}(H^{2}(\Omega),H^{-1}(\Omega))}  \to 0, \quad \mbox{as $\varepsilon \to 0$},
\end{equation}
uniformly for $u\in H^{2}(\Omega)$ such that $\left\| u \right\|_{H^{2}(\Omega)}\leqslant R$.

Now, fix $\frac{1}{2} < s_0 < 1$. Then for any $s$ such that $-1<-s<-s_0<-\frac{1}{2}$,  using interpolation, \eqref{F2} and \eqref{F3}
we have
$$
\begin{array}{lll}
\left\| \dfrac{1}{\varepsilon}\chi_{\omega_{\varepsilon}}Dg_{\Omega}(u)\cdot h_1- Dg_{\Gamma}(u)\cdot h_1\right\|_{H^{-s}(\Omega)} \\

\leqslant \left\| \dfrac{1}{\varepsilon}\chi_{\omega_{\varepsilon}}Dg_{\Omega}(u)\cdot h_1- Dg_{\Gamma}(u)\cdot h_1\right\|^{\theta}_{H^{-s_0}(\Omega)}\left\| \dfrac{1}{\varepsilon}\chi_{\omega_{\varepsilon}}Dg_{\Omega}(u)\cdot h_1- Dg_{\Gamma}(u)\cdot h_1\right\|^{1-\theta}_{H^{-1}(\Omega)}\\

\leqslant  (k_2+k_3)^{\theta} \left\| \dfrac{1}{\varepsilon}\chi_{\omega_{\varepsilon}}Dg_{\Omega}(u)- Dg_{\Gamma}(u)\right\|^{1-\theta}_{\mathcal{L}(H^{2}(\Omega),H^{-1}(\Omega))}\|h_1\|_{H^{2}(\Omega)}, \quad \forall h_1\in H^{2}(\Omega),
\end{array}
$$
for some $0<\theta<1$. Thus using \eqref{F4}, we obtain
$$
\left\| \dfrac{1}{\varepsilon}\chi_{\omega_{\varepsilon}}Dg_{\Omega}(u)- Dg_{\Gamma}(u)\right\|_{\mathcal{L}(H^{2}(\Omega),H^{-s}(\Omega))}  \to 0, \quad \mbox{as $\varepsilon \to 0$},
$$
uniformly for $u\in H^{2}(\Omega)$ such that $\left\| u \right\|_{H^{2}(\Omega)}\leqslant R$.

\vspace{0.2 cm}

\noindent {\it (iii)} Using the item $(ii)$, the hypothesis $w_{\varepsilon}\to w$ in $X$, as $\varepsilon\to 0$, and from Lemma \ref{Lipschitz_01}, we have that there exists $L>0$ independent of $\varepsilon$ such that
\[
\begin{split}
\|DF_{\varepsilon}(w_{\varepsilon})-DF_{0}(w)\|_{\mathcal{L}(X,\mathcal{H})}&\leqslant  \displaystyle \|DF_{\varepsilon}(w_{\varepsilon})-DF_{\varepsilon}(w)\|_{\mathcal{L}(X,\mathcal{H})}+\|DF_{\varepsilon}(w)-DF_{0}(w)\|_{\mathcal{L}(X,\mathcal{H})}\\
&\leqslant \displaystyle  L\|w_{\varepsilon}-w\|_{X}+\|DF_{\varepsilon}(w)-DF_{0}(w)\|_{\mathcal{L}(X,\mathcal{H})}\to 0, \quad \mbox{as $\varepsilon\to0$}.
\end{split}
\]

\noindent {\it (iv)} We take $w_{\varepsilon}\to w$ in $X$, as $\varepsilon\to 0$, and ${\bf h}_{\varepsilon}\to {\bf h}$ in $X$, as $\varepsilon\to 0$. Using the items {\it (i)} and {\it (iii)}, we get
\[
\begin{split}
 \|DF_{\varepsilon}(w_{\varepsilon}) {\bf h}_{\varepsilon}-DF_{0}(w) {\bf h}\|_{\mathcal{H}}&\leqslant \displaystyle  \|DF_{\varepsilon}(w_{\varepsilon}) {\bf h}_{\varepsilon}-DF_{\varepsilon}(w_{\varepsilon}) {\bf h}\|_{\mathcal{H}}+\|DF_{\varepsilon}(w_{\varepsilon}) {\bf h}-DF_{0}(w) {\bf h}\|_{\mathcal{H}}\\
&\leqslant \displaystyle \|DF_{\varepsilon}(w_{\varepsilon})\|_{\mathcal{L}(X,\mathcal{H})}\|{\bf h}_{\varepsilon}-{\bf h}\|_{X}+\|DF_{\varepsilon}(w_{\varepsilon})-DF_{0}(w)\|_{\mathcal{L}(X,\mathcal{H})}\|{\bf h}\|_{X}\\
&\leqslant \displaystyle k\|{\bf h}_{\varepsilon}-{\bf h}\|_{X}+\|DF_{\varepsilon}(w_{\varepsilon})-DF_{0}(w)\|_{\mathcal{L}(X,\mathcal{H})}\|{\bf h}\|_{X}\to0,
\end{split} 
\]
as $\varepsilon\to0$.
\end{proof}


\section{Existence and  upper semicontinuity of attractors}\label{Sec:ExAtt_UpperSem}

From this section onwards we will be assuming all the previous hypotheses. 

In \cite[Section 3]{ABDS}  have been proven that the solutions of the problems \eqref{AP1} and \eqref{AP1b} are globally defined and we can define, for each $\varepsilon \in [0,\varepsilon_0]$, a nonlinear semigroup $\{S_\varepsilon(t):\ t\geqslant 0\}$ in $X$ by
\[
S_\varepsilon(t)w_0=w^\varepsilon(t,w_0),\quad t\geqslant 0,
\]
which it is given by the variation of constants formula
\[
S_\varepsilon(t)w_0=e^{\mathbb{A}t}w_0+\int_0^t e^{\mathbb{A}(t-s)}F_\varepsilon(S_\varepsilon(s)w_0)ds,\quad t\geqslant 0.
\]
Moreover, the semigroups associated to solutions are strongly bounded dissipativite.

To follows, we enunciate the main results obtained in \cite[Section 4]{ABDS}. First, we establish the existence, characterization and uniform boundedness of the global compact attractors for the nonlinear semigroups generated by our problems \eqref{AP1} and \eqref{AP1b}.

\begin{theorem}
For sufficiently small $\varepsilon\geqslant0$. We have:
\\
$(i)$ The parabolic problems \eqref{AP1} and \eqref{AP1b} have a global compact attractor $\mathcal{A}_\varepsilon$ and $\mathcal{A}_\varepsilon=W^u(\mathcal{E}_\varepsilon)$, where
\[
W^u(\mathcal{E}_\varepsilon)=\Big\{w\in X: S_\varepsilon(-t)w\ \mbox{is defined for}\ t\geqslant0\ \mbox{and}\ \lim_{t\to+\infty}\mbox{dist}(S_\varepsilon(-t)w,\mathcal{E}_\varepsilon)=0\},
\]
and $\mathcal{E}_\varepsilon$ denotes the set of equilibria of the problems \eqref{AP1} and \eqref{AP1b}. Moreover, $\mathcal{A}_\varepsilon$ is connected. 
\\
$(ii)$ The union of the global attractors $\bigcup_{\varepsilon\in[0,\varepsilon_0]}\mathcal{A}_\varepsilon$ is a bounded set in $X$.
\end{theorem}
\begin{proof}
For the proof see \cite[Theorems 4 and 5]{ABDS}.
\end{proof}

Also, we establish the convergence of the  nonlinear semigroups as $\varepsilon \to 0$. 

\begin{proposition}\label{proposition4.5}
Under the above hypothesis, let ${\frac12}<s\leqslant 1$ and some fixed $\tau>0$.Then, there exists a function $C(\varepsilon)\geqslant0$ with $C(\varepsilon) \to 0$ as $\varepsilon \to 0$, such that for $w_\varepsilon \in \mathcal{A}_\varepsilon, \ \varepsilon \in (0,\varepsilon_0]$, we have
\[
\Big\| S_{\varepsilon}(t)w_\varepsilon-S_{0}(t)w_\varepsilon \Big\|_{X} \leqslant M(\tau)C(\varepsilon), \quad \forall \ t\in[0,\tau],
\]
for some constant $M(\tau)>0$.
\end{proposition}
\begin{proof}
For the proof see \cite[Proposition 2]{ABDS}.
\end{proof}

Finally, we have the upper semicontinuity of global compact attractors at $\varepsilon=0$.

\begin{theorem}
The family of global attractors $\{\mathcal{A}_\varepsilon: \varepsilon\in [0,\varepsilon_0]\}$ is upper semicontinuous at $\varepsilon=0$; that is,
\[
\operatorname{dist}_{X}(\mathcal{A}_\varepsilon,\mathcal{A}_0) \to 0, \quad \mbox{as} \quad \varepsilon \to 0,
\]
where
\[
\operatorname{dist}_{X}(\mathcal{A}_\varepsilon,\mathcal{A}_0):=\sup_{w_\varepsilon \in\mathcal{A}_\varepsilon}\inf_{w_0 \in\mathcal{A}_0}\{\| w_{\varepsilon}-w_0 \|_{X}\}.
\]
\end{theorem}
\begin{proof}
For the proof see \cite[Theorem 6]{ABDS}.
\end{proof}


\section{Continuity of the set of equilibria}\label{Sec:Cont_Eq}

In order to obtain the lower semicontinuity of global attractors at $\varepsilon=0$ we will need to obtain the continuity of the set of equilibria and then study the continuity of the linearization around each equilibrium. In this section we prove that the family of equilibria $\{\mathcal{E}_\varepsilon: \varepsilon\in[0,\varepsilon_0]\}$ of \eqref{PPrin_1} and \eqref{PPrin_2}  is continuous at $\varepsilon=0$. 

\begin{definition}
The equilibrium solutions of \eqref{PPrin_1} and \eqref{PPrin_2} are those which are independent of time. In other words, for each $\varepsilon\in (0,\varepsilon_0]$, the equilibrium solutions of \eqref{PPrin_1}  are those which are solutions of the elliptic problems
\begin{equation}
\label{eql1}
\begin{cases}
\Delta^2u^\varepsilon+u^\varepsilon=f(u^\varepsilon)+\dfrac{1}{\varepsilon}\chi_{\omega_{\varepsilon}}g(u^\varepsilon)& \mbox{in } \Omega,\\
\dfrac{\partial u^\varepsilon}{\partial\vec{n}}=\dfrac{\partial (\Delta u^\varepsilon)}{\partial\vec{n}}=0 & \mbox{on } \Gamma,
\end{cases}
\end{equation}
and
\begin{equation}
\label{eql2}
\begin{cases}
\Delta \theta^\varepsilon-\theta^\varepsilon=0& \mbox{in}\  \Omega,\\
\dfrac{\partial \theta^\varepsilon}{\partial\vec{n}}=0 & \mbox{on } \Gamma,
\end{cases}
\end{equation}
that is, $\theta^{\varepsilon}$ is identity null in $\Omega$. The equilibrium solutions of \eqref{PPrin_2}  are those which are solutions of the elliptic problems
\begin{equation}\label{eql1BB}
\begin{cases}
\Delta^2u+u=f(u)& \mbox{in } \Omega,\\
\dfrac{\partial u}{\partial\vec{n}}=\dfrac{\partial (\Delta u)}{\partial\vec{n}}=-g(u) & \mbox{on } \Gamma,
\end{cases}
\end{equation}
and
\begin{equation}
\label{eql2BB}
\begin{cases}
\Delta \theta-\theta=0& \mbox{in}\  \Omega,\\
\dfrac{\partial \theta}{\partial\vec{n}}=0 & \mbox{on } \Gamma,
\end{cases}
\end{equation}
that is, $\theta$ is identity null in $\Omega$. 
\end{definition}

\begin{remark}
Equivalently, for each $\varepsilon\in(0,\varepsilon_0]$, the equilibrium solutions of \eqref{AP1} are those which are solutions of the semilinear problems
\begin{equation}
\label{AP1Eq}
\mathbb{A} w^{\varepsilon}+F_\varepsilon(w^{\varepsilon})=0, \quad \mbox{with $w^{\varepsilon}=\displaystyle \Big(\begin{smallmatrix} u^{\varepsilon}\\ 0\\ 0\end{smallmatrix}\Big)$}
\end{equation}
As well as, the equilibrium solutions of \eqref{AP1b} are those which are solutions of the semilinear problem
\begin{equation}
\label{AP2Eq}
\mathbb{A} w+F_0(w)=0, \quad \mbox{with $w=\displaystyle \Big(\begin{smallmatrix} u\\ 0\\ 0\end{smallmatrix}\Big)$}.
\end{equation}
\end{remark}

Thus, the set of equilibria $\mathcal{E}_\varepsilon$ of \eqref{AP1} and \eqref{AP1b}, or equivalently, the set of solutions of \eqref{AP1Eq} and \eqref{AP2Eq} with $\varepsilon\in[0,\varepsilon_0]$, is given by
\[
\mathcal{E}_\varepsilon=\Big\{w^*_{\varepsilon}=\Big(\begin{smallmatrix} u^*_{\varepsilon}\\ 0\\ 0\end{smallmatrix}\Big)\in X:\ \mbox{$u^*_{\varepsilon}$ is solution of \eqref{eql1}} \Big\},\quad \varepsilon\in(0,\varepsilon_0],
\]
and
\[
\mathcal{E}_0=\Big\{w^*=\Big(\begin{smallmatrix} u^*\\ 0\\ 0\end{smallmatrix}\Big)\in X:\ \mbox{$u^*$ is solution of \eqref{eql1BB}} \Big\}.
\]

We will see that each set $\mathcal{E}_{\varepsilon}$ is not empty and it is compact, but for this, we need of the following result
\begin{theorem}
Let $X,Y, Z$ be normed linear spaces, and suppose $T\in \mathcal{L}(X,Y)$ and $S\in \mathcal{L}(Y,Z)$. Then $ST$ is compact, whenever $S$ or $T$ is compact.
\end{theorem}
\begin{proof}
See \cite[Theorem 7.2]{Taylor}.
\end{proof}

\begin{lemma}
\label{compacto}
For each $\varepsilon\in [0,\varepsilon_{0}]$ fixed, the set $\mathcal{E}_{\varepsilon}$ is not empty. Moreover, $\mathcal{E}_{\varepsilon}$ is compact in $X$.
\end{lemma}

\begin{proof}
The bounded linear operator $(\varLambda+I)^{-1}: H^{-s}(\Omega)\to H^{2}(\Omega)$ is compact, because the linear operator $(\varLambda+I)^{-1}: H^{-s}(\Omega)\to H^{4-s}(\Omega)$ is bounded and we have the compact embedding $H^{4-s}(\Omega) \hookrightarrow H^{2}(\Omega)$ for $4-s>2$. Moreover, we have the compact embedding $H^{4}(\Omega) \hookrightarrow H^{2}(\Omega)$ and therefore the bounded linear operator $(\varLambda+I)^{-1}: L^{2}(\Omega)\to H^{2}(\Omega)$ is compact. 

We also have the compact embedding $H^{2}(\Omega) \hookrightarrow L^{2}(\Omega)$ and therefore the bounded linear operator $(\varLambda^{\frac{1}{2}}+I)^{-1}: L^{2}(\Omega)\to L^{2}(\Omega)$ is compact.  Finally, the linear operator $(\varLambda+I)^{-1}(\varLambda^{\frac{1}{2}}+I): H^{2}(\Omega)\to H^{2}(\Omega)$ is compact, because the linear operator $(\varLambda+I)^{-1}(\varLambda^{\frac{1}{2}}+I): H^{2}(\Omega)\to H^{4}(\Omega)$ is bounded and we have the compact embedding $H^{4}(\Omega) \hookrightarrow H^{2}(\Omega)$.
 
Therefore the linear operator $\mathbb{A}^{-1}: \mathcal{H}\to X$ is compact and consequently $\mathbb{A}^{-1}F_{\varepsilon}: X\to X$ is compact.

Now, show that for each $\varepsilon\in [0,\varepsilon_{0}]$ fixed, the set $\mathcal{E}_{\varepsilon}$  is not empty, it is equivalent to show that the compact operator $\mathbb{A}^{-1}F_{\varepsilon}: X\to X$ has at least one fixed point. 

From Lemma \ref{Lipschitz}, we have that there exists $C>0$ independent of $\varepsilon$ such that
\[
\|F_{\varepsilon}(w)\|_{\mathcal{H}}\leqslant C,\quad \mbox{$\forall w\in X$\quad and\quad $\varepsilon \in [0, \varepsilon_{0}]$}.
\]

We consider the closed ball $\bar{B}_{r}(0)$ in $X$, where $r=C\|\mathbb{A}^{-1}\|_{\mathcal{L}(\mathcal{H},X)}$. For each $w\in X$, we have
\begin{eqnarray}
\label{limitacao}
\|\mathbb{A}^{-1}F_{\varepsilon}(w)\|_{X}\leqslant\|\mathbb{A}^{-1}\|_{\mathcal{L}(\mathcal{H},X)} \|F_{\varepsilon}(w)\|_{\mathcal{H}}\leqslant r.
\end{eqnarray}

Therefore, the compact operator $\mathbb{A}^{-1}F_{\varepsilon}: X\to X$ takes $X$ in the ball $\bar{B}_{r}(0)$, in particular, $\mathbb{A}^{-1}F_{\varepsilon}$ takes $\bar{B}_{r}(0)$ into itself. From Schauder Fixed Point Theorem, we obtain that $\mathbb{A}^{-1}F_{\varepsilon}$ has at least one fixed point in $X$. 

Now, for each $\varepsilon\in [0,\varepsilon_{0}]$ fixed, we will prove that $\mathcal{E}_{\varepsilon}$ is compact in $X$. For each $\varepsilon\in [0,\varepsilon_{0}]$ fixed, let $\{w^{*}_{\varepsilon,n}\}_{n\in \mathbb{N}}$ be a sequence in $\mathcal{E}_{\varepsilon}$, then $w^{*}_{\varepsilon,n}=-\mathbb{A}^{-1}F_{\varepsilon}(w^{*}_{\varepsilon,n})$, for all $n\in \mathbb{N}$. Similarly to (\ref{limitacao}), we get that $\{w^{*}_{\varepsilon,n}\}_{n\in \mathbb{N}}$ is a bounded sequence in $X$. Thus, for each $\varepsilon\in [0,\varepsilon_{0}]$ fixed, $\{-\mathbb{A}^{-1}F_{\varepsilon}(w^{*}_{\varepsilon,n})\}_{n\in \mathbb{N}}$ has a convergent subsequence, that we will denote by $\{-\mathbb{A}^{-1}F_{\varepsilon}(w^{*}_{\varepsilon,n_{k}})\}_{k\in \mathbb{N}}$, with limit $w^{*}_\varepsilon\in X$, that is,
\[
-\mathbb{A}^{-1}F_{\varepsilon}(w^{*}_{\varepsilon,n_{k}})\to w^{*}_\varepsilon \quad \mbox{in}\quad X,\quad \mbox{as $k\to \infty$}.
\]
Hence, $w^{*}_{\varepsilon,n_{k}}\to w^{*}_\varepsilon$ in $X$, as $k\to \infty$. 

By continuity of operator $\mathbb{A}^{-1}F_{\varepsilon}:X\to X$, we get 
\[
-\mathbb{A}^{-1}F_{\varepsilon}(w^{*}_{\varepsilon,n_{k}})\to -\mathbb{A}^{-1}F_{\varepsilon}(w^{*}_\varepsilon)\quad \mbox{in}\quad X,\quad \mbox{as $k\to \infty$}.
\]

By the uniqueness of the limit, $w^{*}_{\varepsilon}=-\mathbb{A}^{-1}F_{\varepsilon}(w^{*}_{\varepsilon})$. Thus, $\mathbb{A} w^{*}_{\varepsilon}+F_{\varepsilon}(w^{*}_{\varepsilon})=0$ and $w^{*}_{\varepsilon}\in \mathcal{E}_{\varepsilon}$. Therefore, $\mathcal{E}_{\varepsilon}$ is a compact set in $X$. 
\end{proof}

The upper semicontinuity of the family  $\{\mathcal{E}_\varepsilon: \varepsilon\in[0,\varepsilon_0]\}$ at $\varepsilon=0$ is a consequence of the upper semicontinuity of attractors at $\varepsilon=0$.

\begin{theorem}\label{theoremUS}
The family  $\{\mathcal{E}_\varepsilon: \varepsilon\in[0,\varepsilon_0]\}$ is upper semicontinuous  at $\varepsilon=0$.
\end{theorem}

\begin{proof}
Initially, we observe that $\mathcal{E}_\varepsilon\subset \mathcal{A}_\varepsilon$ for any $\varepsilon\in[0,\varepsilon_0]$, and therefore, $\mathcal{E}_\varepsilon$ is bounded in $X$. We will prove that for any sequence of $\varepsilon \to 0$ and for any $w^{*}_{\varepsilon} \in \mathcal{E}_{\varepsilon}$ we can extract a subsequence which converges to an element of $\mathcal{E}_{0}$. From the upper semicontinuity of the attractors and using that $w^{*}_{\varepsilon} \in \mathcal{E}_{\varepsilon} \subset \mathcal{A}_\varepsilon$, we can extract a subsequence $w^{*}_{\varepsilon_k} \in \mathcal{E}_{\varepsilon_k}$ with $\varepsilon_k \to 0$, as $k\to \infty$, and we obtain the existence of a $w^{*} \in \mathcal{A}_0$ such that 
\[
\|w^{*}_{\varepsilon_k}-w^{*}\|_{X} \to 0, \quad \mbox{as $k\to \infty$}.
\]
We need to prove that $w ^{*}\in \mathcal{E}_{0}$; that is, $S_{0}(t)w^{*}=w^{*}$, for any $t\geqslant 0$.

We first observe that for any $t>0$,
\[
\|w^{*}_{\varepsilon_k}-S_{0}(t)w^{*}\|_{X}\leqslant \|w^{*}_{\varepsilon_k}-w^{*}\|_{X}+\|w^{*}-S_{0}(t)w^{*}\|_{X}\to \|w^{*}-S_{0}(t)w^{*}\|_{X}, \quad \mbox{as $k\to \infty$}.
\]
Moreover, for a fixed $\tau >0$ and for any $t\in (0,\tau)$, we obtain
\[
\begin{split}
\|w^{*}_{\varepsilon_k}-S_{0}(t)w^{*}\|_{X}&=\|S_{\varepsilon_k}(t)w^{*}_{\varepsilon_k}-S_{0}(t)w^{*}\|_{X}\\
&\leqslant \|S_{\varepsilon_k}(t)w^{*}_{\varepsilon_k}-S_{0}(t)w^{*}_{\varepsilon_k}\|_{X}+\|S_{0}(t)w^{*}_{\varepsilon_k}-S_{0}(t)w^{*} \|_{X} \to 0,\quad \mbox{as $k\to \infty$,}
\end{split}
\]
where we have used the continuity of semigroups given by Proposition \ref{proposition4.5}. In particular, we have that for each $t\geqslant 0$, $S_{0}(t)w^{*}=w^{*}$, which implies that $w^{*}\in \mathcal{E}_0$.
\end{proof}

The proof of lower semicontinuity requires additional assumptions. We need to assume that the equilibrium solutions of \eqref{AP2Eq} are stable under perturbation, this stability under perturbation will be given by the hyperbolicity.

\begin{definition}\label{def_hyperbolic}
We say that the solution $w^{*}$ of \eqref{AP2Eq} is hyperbolic if the spectrum $\sigma(\mathbb{A}+DF_0(w^{*}))$ of $\mathbb{A}+DF_0(w^{*})$ is disjoint from the imaginary axis. 
\end{definition}

\begin{theorem} \label{theoremisolated}
If all solutions of \eqref{AP2Eq} are isolated, then there are only a finite number of them. Any hyperbolic solution of \eqref{AP2Eq} is isolated.
\end{theorem}
\begin{proof}
Since $\mathcal{E}_0$ is compact we only need to prove that hyperbolic solution is isolated. We note that $w^{*}\in \mathcal{E}_0$ is a solution of \eqref{AP2Eq} if and only if $w^{*}$ is a fixed point of
\[
T(\xi):=-(\mathbb{A}+DF_0(w^{*}))^{-1}(F_0(\xi)-DF_0(w^{*})\xi).
\]
It is not difficult to see that there is $\delta>0$ such that $T$ is a contraction map from closed ball centered at $w^{*}$ and of radius $\delta$ in $X$, $\overline{B}_{\delta}(w^{*})$, into itself. Thus we obtain that $w^{*}$ is the only element in  $\mathcal{E}_0$ in the ball  $\overline{B}_{\delta}(w^{*})$.
\end{proof}

\begin{lemma}\label{lemaass2}
Let $w^{*}\in X$. Then, for each $\varepsilon\in[0,\varepsilon_0]$ fixed, the operator $\mathbb{A}^{-1}DF_{\varepsilon}(w^{*}): X\to X$ is compact. For any bounded family $\left\{w_{\varepsilon}\right\}_{\varepsilon\in(0,\varepsilon_{0}]}$ in $X$, the family $\{\mathbb{A}^{-1}DF_{\varepsilon}(w^{*})w_{\varepsilon}\}_{\varepsilon\in(0,\varepsilon_{0}]}$ is relatively compact in $X$. Moreover, if  $w_{\varepsilon}\to w$ in $X$, as $\varepsilon\to0$, then 
\[
\mathbb{A}^{-1}DF_{\varepsilon}(w^{*})w_{\varepsilon}\to \mathbb{A}^{-1}DF_{0}(w^{*})w \quad \mbox{in}\quad X,  \quad \mbox{as $\varepsilon\to0$}.
\]
\end{lemma}

\begin{proof}
For each $\varepsilon\in[0,\varepsilon_{0}]$ fixed, the compactness of linear operator $\mathbb{A}^{-1}DF_{\varepsilon}(w^{*}): X\to X$ follows from item {\it (i)} of Lemma \ref{resultsconvnonlinearity} and of compactness of linear operator 
\[
\mathbb{A}^{-1}: \mathcal{H}\to X.
\]

Let $\{w_{\varepsilon}\}_{\varepsilon\in(0,\varepsilon_{0}]}$ be a bounded family in $X$. Since
$$
\|DF_{\varepsilon}(w^{*})w_{\varepsilon}\|_{ \mathcal{H}}\leqslant \|DF_{\varepsilon}(w^{*})\|_{\mathcal{L}(X, \mathcal{H})}\left\|w_{\varepsilon}\right\|_{X}, \quad \mbox{$\forall \varepsilon\in (0,\varepsilon_{0}]$},
$$
and from item {\it (i)} of Lemma \ref{resultsconvnonlinearity}, $\{DF_{\varepsilon}(w^{*})\}_{\varepsilon\in (0,\varepsilon_{0}]}$ is a bounded family in $\mathcal{L}(X, \mathcal{H})$, uniformly in $\varepsilon$, then $\{DF_{\varepsilon}(w^{*})w_{\varepsilon}\}_{\varepsilon\in (0,\varepsilon_{0}]}$ is a bounded family in $ \mathcal{H}$. By compactness of  the linear operator $\mathbb{A}^{-1}: \mathcal{H}\to X$, we have that $\{\mathbb{A}^{-1}DF_{\varepsilon}(w^{*})w_{\varepsilon}\}_{\varepsilon\in (0,\varepsilon_{0}]}$ has a convergent subsequence in $X$. Therefore, the family $\{\mathbb{A}^{-1}DF_{\varepsilon}(w^{*})w_{\varepsilon}\}_{\varepsilon\in (0,\varepsilon_{0}]}$ is relatively compact.

Now, let us take $w_{\varepsilon}\to w$ in $X$, as $\varepsilon\to0$. Thus, from item {\it (iv)} of Lemma \ref{resultsconvnonlinearity},
$$
DF_{\varepsilon}(w^{*})w_{\varepsilon}\to DF_{0}(w^{*})w \quad \mbox{in}\quad  \mathcal{H}, \quad \mbox{as $\varepsilon\to0$}.
$$
  
By continuity of the linear operator $\mathbb{A}^{-1}: \mathcal{H}\to X$, we conclude that 
\[
\mathbb{A}^{-1}DF_{\varepsilon}(w^{*})w_{\varepsilon}\to \mathbb{A}^{-1}DF_{0}(w^{*})w \quad \mbox{in}\quad X, \quad \mbox{as $\varepsilon\to0$}.
\]
\end{proof}

\begin{lemma}\label{important}
Let $w^{*} \in X$ such that $0\not\in\mbox{Re}\sigma(\mathbb{A}+DF_0(w^{*}))$. Then, there exist $\varepsilon_0>0$ and $C>0$ independent of $\varepsilon$ such that $0\not\in \mbox{Re}\sigma(\mathbb{A}+DF_{\varepsilon}(w^{*}))$ and 
\begin{equation}\label{inverse}
\|(\mathbb{A}+DF_{\varepsilon}(w^{*}))^{-1}\|_{\mathcal{L}(\mathcal{H},X)} \leqslant C, \quad \forall \varepsilon \in [0,\varepsilon_0].
\end{equation}
Furthermore, for each $\varepsilon \in [0,\varepsilon_0]$ fixed, the operator $(\mathbb{A}+DF_{\varepsilon}(w^{*}))^{-1}:\mathcal{H} \to X$ is compact. For any bounded family $\{w_{\varepsilon}\}_{\varepsilon \in (0,\varepsilon_0]}$ in $\mathcal{H}$, the family \linebreak$\{(\mathbb{A}+DF_{\varepsilon}(w^{*}))^{-1}w_{\varepsilon}\}_{\varepsilon \in (0,\varepsilon_0]}$ is relatively compact in $X$. Moreover, if $w_{\varepsilon} \to w$ in $\mathcal{H}$, as $\varepsilon \to 0$, then
\[
(\mathbb{A}+DF_{\varepsilon}(w^{*}))^{-1}w_{\varepsilon} \to (\mathbb{A}+DF_{0}(w^{*}))^{-1}w \quad \mbox{in} \quad X, \quad \mbox{as} \quad \varepsilon \to 0.
\]
\end{lemma}

\begin{proof}
First, for each $\varepsilon \in [0,\varepsilon_0]$, we note
\[
(\mathbb{A}+DF_{\varepsilon}(w^{*}))^{-1}=[\mathbb{A}(I+\mathbb{A}^{-1}DF_{\varepsilon}(w^{*}))]^{-1}=(I+\mathbb{A}^{-1}DF_{\varepsilon}(w^{*}))^{-1}\mathbb{A}^{-1}.
\]
Then, prove that $0\not\in\mbox{Re}\sigma(\mathbb{A}+DF_{\varepsilon}(w^{*}))$ it is equivalent to prove that $1 \in \rho(\mathbb{A}^{-1}DF_{\varepsilon}(w^{*}))$. Moreover, to prove that there exist $\varepsilon_0>0$ and $C>0$ independent of $\varepsilon$ such that
\eqref{inverse} holds, it is enough to prove that there exist $\varepsilon_0>0$ and $K>0$ independent of $\varepsilon$ such that

\begin{equation}\label{inverse1}
\|(I+\mathbb{A}^{-1}DF_{\varepsilon}(w^{*}))^{-1}\|_{\mathcal{L}(X)} \leqslant K, \quad \forall \varepsilon \in [0,\varepsilon_0].
\end{equation}
Indeed, we note that 
\[
\begin{split}
\|(\mathbb{A}+DF_{\varepsilon}(w^{*}))^{-1}\|_{\mathcal{L}(\mathcal{H},X)} & \leqslant \|(I+\mathbb{A}^{-1}DF_{\varepsilon}(w^{*}))^{-1}\|_{\mathcal{L}(X)} \|\mathbb{A}^{-1}\|_{\mathcal{L}(\mathcal{H},X)} \\
&= K \|\mathbb{A}^{-1}\|_{\mathcal{L}(\mathcal{H},X)} =C,
\end{split}
\]
for all $\varepsilon \in [0,\varepsilon_0]$.

Then we will show \eqref{inverse1}. From hypothesis $0\not\in\mbox{Re}\sigma(\mathbb{A}+DF_0(w^{*}))$ then $1 \in \rho(\mathbb{A}^{-1}DF_{0}(w^{*}))$. Thus, there exists the inverse 
\[
(I+\mathbb{A}^{-1}DF_{\varepsilon}(w^{*}))^{-1}: X \to X
\]
and, particular we have $N(I+\mathbb{A}^{-1}DF_{0}(w^{*}))=\{0\}$.

For simplicity of notation, let $J_{\varepsilon}=\mathbb{A}^{-1}DF_{\varepsilon}(w^{*})$, for all $\varepsilon \in [0,\varepsilon_0]$. From Lemma \ref{lemaass2} we have that, for each $\varepsilon \in [0,\varepsilon_0]$ fixed, the operator $J_{\varepsilon}: X \to X$ is compact. Using the compactness of $J_{\varepsilon}$ we will show that \eqref{inverse1} hold, if and only if,
\begin{equation}\label{operatorJ}
\|(I+J_{\varepsilon})z_{\varepsilon}\|_{X} \geqslant \frac{1}{K}, \quad \forall \varepsilon \in [0,\varepsilon_0] \quad \mbox{and} \quad  \|z_{\varepsilon}\|_{X}=1.
\end{equation} 

Indeed, suppose that \eqref{inverse1} holds, then there exists the inverse $(I+J_{\varepsilon})^{-1}: X \to X$ and it is continuous. Moreover, 
\[
\|(I+J_{\varepsilon})^{-1}y_{\varepsilon}\|_{X} \leqslant K \|y_{\varepsilon}\|_{X}, \quad \forall  \varepsilon \in [0,\varepsilon_0]\quad  \mbox{and} \quad \forall y_{\varepsilon} \in X.
\]

Now if $z_{\varepsilon} \in X$ is such that $\|z_{\varepsilon}\|_{X}=1$ and taking $y_{\varepsilon}=(I+J_{\varepsilon})z_{\varepsilon}$, we have
\[
\|(I+J_{\varepsilon})^{-1}(I+J_{\varepsilon})z_{\varepsilon}\|_{X} \leqslant K \|(I+J_{\varepsilon})z_{\varepsilon}\|_{X} 
\]
and
\[
1=\|z_{\varepsilon}\|_{X}\leqslant  K \|(I+J_{\varepsilon})z_{\varepsilon}\|_{X},
\]
in other words, 
\[
\|(I+J_{\varepsilon})z_{\varepsilon}\|_{X} \geqslant \frac{1}{K}.
\]

On the other hand, suppose that \eqref{operatorJ} holds. We will show that there exists the inverse $(I+J_{\varepsilon})^{-1}: X \to X$, it is continuous and satisfies \eqref{inverse1}. From \eqref{operatorJ}, we obtain the following estimative 
\begin{equation}\label{operatorJ1}
\|(I+J_{\varepsilon})z_{\varepsilon}\|_{X} \geqslant \frac{1}{K}\|z_{\varepsilon}\|_{X}, \quad \forall \varepsilon \in [0,\varepsilon_0] \quad \mbox{and} \quad  \forall z_{\varepsilon}\in X.
\end{equation} 

Now, let $z_{\varepsilon}\in X$ such that $(I+J_{\varepsilon})z_{\varepsilon}=0$. From \eqref{operatorJ1} follows $z_{\varepsilon}=0$. Thus, for each $\varepsilon \in [0,\varepsilon_0], \ N(I+J_{\varepsilon})=\{0\}$ and the operator $I+J_{\varepsilon}$ is injective. Since there exists the inverse $(I+J_{\varepsilon})^{-1}:R(I+J_{\varepsilon}) \to X$ and $J_{\varepsilon}$ is compact, then by Fredhlom Alternative Theorem, we have
\[
N(I+J_{\varepsilon})=\{0\}  \Leftrightarrow  R(I+J_{\varepsilon})=X.
\]
Then $I+J_{\varepsilon}$  is bijective, thus there exists the inverse $(I+J_{\varepsilon})^{-1}; X \to X$.

Now, taking $y_{\varepsilon} \in X$ there exists $z_{\varepsilon} \in X$ such that $y_{\varepsilon}=(I+J_{\varepsilon})z_{\varepsilon}$ and $z_{\varepsilon}=(I+J_{\varepsilon})^{-1}y_{\varepsilon}$. From \eqref{operatorJ1} we have 
\[
\|(I+J_{\varepsilon})^{-1}y_{\varepsilon}\|_{X}=\|z_{\varepsilon}\|_{X} \leqslant K \|(I+J_{\varepsilon})z_{\varepsilon}\|_{X}=K \|y_{\varepsilon}\|_{X}
\]
and
\[
\|(I+J_{\varepsilon})^{-1}\|_{\mathcal{L}(X)} \leqslant K, \quad \forall \varepsilon \in [0,\varepsilon_0],
\]
and thus \eqref{inverse1} holds.

Therefore \eqref{inverse1} and \eqref{operatorJ} are equivalents, then we will show \eqref{operatorJ}. Suppose that \eqref{operatorJ} is not true, that is, there exists a sequence $\{z_{n}\}_{n \in \mathbb{N}}$ in $X$, with $\|z_{n}\|_{X}=1$ and $\varepsilon_{n} \to 0$, as $n \to \infty$, such that
\[
\|(I+J_{\varepsilon_n})z_{n}\|_{X} \to 0, \quad \mbox{as} \quad n \to \infty.
\]

From Lemma \ref{lemaass2} we get that $\{J_{\varepsilon_{n}}z_{n}\}_{n \in \mathbb{N}}$ is relatively compact. Thus, $\{J_{\varepsilon_{n}}z_{n}\}_{n \in \mathbb{N}}$ has a convergent subsequence, which still we denote by $\{J_{\varepsilon_{n}}z_{n}\}_{n \in \mathbb{N}}$, with limit $z \in X$, that is,
 \[
J_{\varepsilon_{n}} z_{n} \to z \quad \mbox{in} \quad X, \quad \mbox{as} \quad n \to \infty.
\]

Since $z_{n}+J_{\varepsilon_{n}} z_{n} \to 0$ in $X$, as $n \to \infty$, then $z_{n} \to -z$ in $X$, as $n \to \infty$ and thus $\|z\|_{X}=1$. Moreover, using the Lemma \ref{lemaass2} we get $J_{\varepsilon_{n}}z_{n} \to -J_{0}z$ as $n \to \infty$. Then,
\[
z_n + J_{\varepsilon_{n}} z_n \to -(z + J_{0} z) \quad \mbox{in} \quad X, \quad \mbox{as} \quad n\to \infty.
\]
By uniqueness of the limit, $(I+J_{0})z=0$, with $z\neq0$, contradicting the fact of the operator $I+J_{0}$ be injective, because $0\not\in\mbox{Re}\sigma(\mathbb{A}+DF_0(w^{*}))$. Showing that \eqref{operatorJ} holds. With this we conclude that there exist $\varepsilon_{0}>0$ and $C>0$ independent of $\varepsilon$ such that \eqref{inverse} holds.

Now, for each $\varepsilon\in[0,\varepsilon_{0}]$, the operator $(\mathbb{A}+DF_{\varepsilon}(w^{*}))^{-1}$ is compact and the prove of this compactness follows similarly to account below.  

Let $\{w_{\varepsilon}\}_{\varepsilon\in(0,\varepsilon_{0}]}$ be a bounded family in $\mathcal{H}$. For each $\varepsilon\in(0,\varepsilon_{0}]$, let  $\vartheta_{\varepsilon}=(\mathbb{A}+DF_{\varepsilon}(w^{*}))^{-1}w_{\varepsilon}.$ From \eqref{inverse} we have
\[
\begin{split}
\displaystyle \left\|\vartheta_{\varepsilon}\right\|_{X}&\leqslant \|(\mathbb{A}+DF_{\varepsilon}(w^{*}))^{-1}w_{\varepsilon}\|_{X} \\
&\leqslant \|(\mathbb{A}+DF_{\varepsilon}(w^{*}))^{-1}\|_{\mathcal{L}(\mathcal{H},X)}\left\|w_{\varepsilon}\right\|_{\mathcal{H}}\\
&\leqslant C\left\|w_{\varepsilon}\right\|_{\mathcal{H}}.
\end{split}
\]

Hence, $\{\vartheta_{\varepsilon}\}_{\varepsilon\in(0,\varepsilon_{0}]}$ is a bounded family in $X$. Moreover, 
\[
\vartheta_{\varepsilon}=(\mathbb{A}+DF_{\varepsilon}(w^{*}))^{-1}w_{\varepsilon}=(I+\mathbb{A}^{-1}DF_{\varepsilon}(w^{*}))^{-1}\mathbb{A}^{-1}w_{\varepsilon}
\]
in other words,
\[
(I+\mathbb{A}^{-1}DF_{\varepsilon}(w^{*}))\vartheta_{\varepsilon}=\mathbb{A}^{-1}w_{\varepsilon},
\]
and equivalently,
\[
\vartheta_{\varepsilon}=-\mathbb{A}^{-1}DF_{\varepsilon}(w^{*})\vartheta_{\varepsilon}+\mathbb{A}^{-1}w_{\varepsilon}.
\]

By compactness of $\mathbb{A}^{-1}:\mathcal{H}\to X$, we get that $\{\mathbb{A}^{-1}w_{\varepsilon}\}_{\varepsilon\in(0,\varepsilon_{0}]}$ has a convergent subsequence in $X$. Moreover, using the Lemma \ref{lemaass2}, we have that $\{\mathbb{A}^{-1}DF_{\varepsilon}(w^{*})\vartheta_{\varepsilon}\}_{\varepsilon\in(0,\varepsilon_{0}]}$ is relatively compact in $X$, then $\{\mathbb{A}^{-1}DF_{\varepsilon}(w^{*})\vartheta_{\varepsilon}\}_{\varepsilon\in(0,\varepsilon_{0}]}$ has a convergent subsequence in $X$. Therefore, $\{\vartheta_{\varepsilon}\}_{\varepsilon\in(0,\varepsilon_{0}]}$ has a convergent subsequence in $X$, that is, the family  $\{(\mathbb{A}+DF_{\varepsilon}(w^{*}))^{-1}w_{\varepsilon}\}_{\varepsilon\in(0,\varepsilon_{0}]}$ has a convergent subsequence in $X$, thus it is relatively compact in $X$.

Now, we take $w_{\varepsilon}\to w$ in $\mathcal{H}$, as $\varepsilon\to0$. By continuity of operator $\mathbb{A}^{-1}:\mathcal{H}\to X$, we have
\[
\mathbb{A}^{-1}w_{\varepsilon}\to \mathbb{A}^{-1}w \quad \mbox{in}\quad X, \quad \mbox{as $\varepsilon\to0$}.
\]
Moreover, $\{w_{\varepsilon}\}_{\varepsilon\in(0,\varepsilon_{0}]}$ is bounded in $\mathcal{H}$, for some $\varepsilon_{0}>0$ sufficiently small, and we have that from the above that $\{\vartheta_{\varepsilon}\}_{\varepsilon\in(0,\varepsilon_{0}]}$, with $\varepsilon_{0}>0$ sufficiently small, has a convergent subsequence, which we again denote by $\{\vartheta_{\varepsilon}\}_{\varepsilon\in(0,\varepsilon_{0}]}$, with limit $\vartheta\in X$, that is, 
\[
\vartheta_{\varepsilon}\to \vartheta\quad \mbox{in}\quad X,\quad \mbox{as $\varepsilon\to0$}.
\]

From Lemma \ref{lemaass2} we get
\[
\mathbb{A}^{-1}DF_{\varepsilon}(w^{*})\vartheta_{\varepsilon}\to \mathbb{A}^{-1}DF_{0}(w^{*})\vartheta\quad \mbox{in}\quad X, \quad \mbox{as $\varepsilon\to0$}.
\] 

Thus, $\vartheta$ satisfies $\vartheta=-\mathbb{A}^{-1}DF_{0}(w^{*})\vartheta+\mathbb{A}^{-1}w$, and so $\vartheta=(\mathbb{A}+DF_{0}(w^{*}))^{-1}w$. Therefore,
\[
(\mathbb{A}+DF_{\varepsilon}(w^{*}))^{-1}w_{\varepsilon}\to (\mathbb{A}+DF_{0}(w^{*}))^{-1}w \quad \mbox{in}\quad X,\quad \mbox{as $\varepsilon\to0$}.
\]
The limit above is independent of the subsequence, thus whole family $\{(\mathbb{A}+DF_{\varepsilon}(w^{*}))^{-1}w_{\varepsilon}\}_{\varepsilon\in(0,\varepsilon_{0}]}$ converges to $(\mathbb{A}+DF_{0}(w^{*}))^{-1}w$ in $X$, as $\varepsilon\to 0$. 
\end{proof}

\begin{theorem}\label{Teoemsas}
Suppose that $w^{*}$ is a solution for \eqref{AP2Eq} and that $0\not\in\mbox{Re}\sigma(\mathbb{A}+DF_0(w^{*}))$. Then there are $\varepsilon_0>0$ and $\delta>0$ such that the problem \eqref{AP1Eq} has exactly one solution, $w^{*}_\varepsilon$, in the closed ball centered at $w^{*}$ and radius $\delta$, $\{\xi\in X:\|\xi-w^{*}\|_X\leqslant\delta\}$, for any $\varepsilon\in(0,\varepsilon_0]$. Futhermore, 
\[
\|w^{*}_\varepsilon-w^{*}\|_X\to0,\quad\mbox{as } \varepsilon\to0.
\]
\end{theorem}

\begin{proof}
Initially, note that from Lemma \ref{important} there exists $\varepsilon_0 >0$ and $C>0$, independent of $\varepsilon_0$, such that
\begin{equation}
\label{limit}
\|(\mathbb{A}+DF_{\varepsilon}(w^{*}))^{-1}\|_{\mathcal{L}(\mathcal{H},X)} \leqslant C, \quad \forall \varepsilon \in (0,\varepsilon_0].
\end{equation}
We note that if $w_{\varepsilon}, \varepsilon \in (0,\varepsilon_0]$, is a solution of \eqref{AP1Eq}, then
\[
0=(\mathbb{A}+DF_{\varepsilon}(w^{*}))[w_{\varepsilon}+(\mathbb{A}+DF_{\varepsilon}(w^{*}))^{-1}(F_{\varepsilon}(w_{\varepsilon})-DF_{\varepsilon}(w^{*})w_{\varepsilon})].
\]

Since $(\mathbb{A}+DF_{\varepsilon}(w^{*}))$ is invertible, then $w_{\varepsilon}$ is a solution of \eqref{AP1Eq} if and only if $w_{\varepsilon}$ is a fixed point of the map $T_{\varepsilon}:X \to X$ defined by
\[
T_{\varepsilon}(w_{\varepsilon})=-(\mathbb{A}+DF_{\varepsilon}(w^{*}))^{-1}(F_{\varepsilon}(w_{\varepsilon})-DF_{\varepsilon}(w^{*})w_{\varepsilon}).
\]
We have that 
\begin{equation}\label{conv}
T_{\varepsilon}(w^{*}) \to w^{*} \quad \mbox{in} \quad X \quad \mbox{as} \quad \varepsilon \to 0.
\end{equation}
In fact, using \eqref{limit}, item \textit{(iii)} of Lemma \ref{Lipschitz}, item \textit{(iv)} of Lemma \ref{resultsconvnonlinearity} and Lemma \ref{important}, for $\varepsilon \in (0,\varepsilon_0]$, we have
\[
\begin{split}
&\|T_{\varepsilon}(w^{*})-w^{*}\|_{X} =\|T_{\varepsilon}(w^{*})-T(w^{*})\|_{X} \\
&=\|-(\mathbb{A}+DF_{\varepsilon}(w^{*}))^{-1}(F_{\varepsilon}(w^{*})-DF_{\varepsilon}(w^{*})w^{*})+ (\mathbb{A}+DF_{0}(w^{*}))^{-1}(F_{0}(w^{*})-DF_{0}(w^{*})w^{*})\|_{X} \\
& \leqslant \|-(\mathbb{A}+DF_{\varepsilon}(w^{*}))^{-1}[F_{\varepsilon}(w^{*})-DF_{\varepsilon}(w^{*})w^{*})-(F_{0}(w^{*})-DF_{0}(w^{*})w^{*})]\|_{X} \\
&+ \|[(\mathbb{A}+DF_{\varepsilon}(w^{*}))^{-1}-(\mathbb{A}+DF_{0}(w^{*}))^{-1}](DF_{0}(w^{*})w^{*}-F_{0}(w^{*}))\|_{X} \\
& \leqslant C(\|F_{\varepsilon}(w^{*})-F_{0}(w^{*})\|_{\mathcal{H}} + \|DF_{\varepsilon}(w^{*})w^{*}-DF_{0}(w^{*})w^{*}\|_{\mathcal{H}}) \\
& + \|[(\mathbb{A}+DF_{\varepsilon}(w^{*}))^{-1}-(\mathbb{A}+DF_{0}(w^{*}))^{-1}](DF_{0}(w^{*})w^{*}-F_{0}(w^{*}))\|_{X} \to 0, \quad \mbox{as $\varepsilon \to 0$}.
\end{split}
\]

Next we prove that there exists $\delta >0$ and that for $\varepsilon \in (0,\varepsilon_0]$ the map $T_{\varepsilon}$ is contraction
from 
\[
\bar{B}_{\delta}(w^{*})=\{\xi \in X : \|\xi-w^{*}\|_{X} \leqslant \delta\}
\]
into itself, uniformly in $\varepsilon$. First note that from Lemma \ref{Lipschitz_01} there exist $\tilde{\delta}=\tilde{\delta}(C)>0$ independent of $\varepsilon$ such that
\begin{equation}\label{dif}
C\|F_{\varepsilon}(w_{\varepsilon})-F_{\varepsilon}(z_{\varepsilon})-DF_{\varepsilon}(w^{*})(w_{\varepsilon}-z_{\varepsilon})\|_{\mathcal{H}} \leqslant \frac{1}{2} \|w_{\varepsilon}-z_{\varepsilon}\|_{X}, \quad \forall \varepsilon \in (0,\varepsilon_0],
\end{equation}
for $\|w_{\varepsilon}-z_{\varepsilon}\|_{X} \leqslant \tilde{\delta}$.

We take $\delta=\frac{\tilde{\delta}}{2}$ and let $w_{\varepsilon},\ z_{\varepsilon} \in \bar{B}_{\delta}(w^{*})$ and using \eqref{limit} and \eqref{dif}, for $\varepsilon \in (0,\varepsilon_0]$ we have
\[
\begin{split}
\|T_{\varepsilon}(w_{\varepsilon})-T_{\varepsilon}(z_{\varepsilon})\|_{X} & = \|-(\mathbb{A}+DF_{\varepsilon}(w^{*}))^{-1}(F_{\varepsilon}(w_{\varepsilon})-F_{\varepsilon}(z_{\varepsilon})-DF_{\varepsilon}(w^{*})(w_{\varepsilon}-z_{\varepsilon})\|_{X} \\
& \leqslant C \|F_{\varepsilon}(w_{\varepsilon})-F_{\varepsilon}(z_{\varepsilon})-DF_{\varepsilon}(w^{*})(w_{\varepsilon}-z_{\varepsilon})\|_{\mathcal{H}} \\
& \leqslant \frac{1}{2} \|w_{\varepsilon}-z_{\varepsilon}\|_{X}.
\end{split}
\]
To show that $T_{\varepsilon}(\bar{B}_{\delta}(w^{*})) \subset \bar{B}_{\delta}(w^{*})$, observe that if $w_{\varepsilon} \in \bar{B}_{\delta}(w^{*})$ and from \eqref{conv} there is $\varepsilon_0$ such that $\|T_{\varepsilon}(w^{*})-w^{*}\|_{X}\leqslant \frac{\delta}{2}$, then
\[
\begin{split}
\|T_{\varepsilon}(w_{\varepsilon})-w^{*}\|_{X} & \leqslant \|T_{\varepsilon}(w_{\varepsilon})-T_{\varepsilon}(w^{*})\|_{X} + \|T_{\varepsilon}(w^{*})-w^{*}\|_{X} \\
& \leqslant \frac{1}{2} \|w_{\varepsilon}-w^{*}\|_{X} + \|T_{\varepsilon}(w^{*})-w^{*}\|_{X} \\
&\leqslant \frac{\delta}{2}+\frac{\delta}{2}=\delta.
\end{split}
\]
Therefore, $T_{\varepsilon}:\bar{B}_{\delta}(w^{*}) \to \bar{B}_{\delta}(w^{*})$ is a contraction, for all $\varepsilon \in (0,\varepsilon_0]$, 
and then by Contraction Theorem  there is only one point fixed $w^{*}_{\varepsilon}$ of $T_{\varepsilon}$ in $\bar{B}_{\delta}(w^{*})$.

Now we will show that $w^{*}_{\varepsilon} \to w^{*}$ in $X$ as $\varepsilon \to 0$. In fact, 
\[
\begin{split}
\|w^{*}_{\varepsilon}-w^{*}\|_{X} &= \|T_{\varepsilon}(w^{*}_{\varepsilon})-w^{*}\|_{X} \leqslant \|T_{\varepsilon}(w^{*}_{\varepsilon})-T_{\varepsilon}(w^{*})\|_{X} + \|T_{\varepsilon}(w^{*})-w^{*}\|_{X} \\
& \leqslant \frac{1}{2} \|w^{*}_{\varepsilon}-w^{*}\|_{X} + \|T_{\varepsilon}(w^{*})-w^{*}\|_{X}.
\end{split}
\]
Thus, using again \eqref{conv} we have
\[
\|w^{*}_{\varepsilon}-w^{*}\|_{X} \leqslant 2  \|T_{\varepsilon}(w^{*})-w^{*}\|_{X} \to 0, \quad \mbox{as} \quad \varepsilon \to 0.
\]
\end{proof}

\begin{remark}
The Theorem \ref{theoremUS} and the Theorem \ref{Teoemsas} show the continuity of the set of equilibria $\mathcal{E}_{\varepsilon}, \ \varepsilon \in [0,\varepsilon_0]$ at $\varepsilon=0$; namely, the Theorem \ref{Teoemsas} shows the lower semicontinuity of the set of equilibria. Moreover, the Theorem \ref{Teoemsas} shows that if $w^{*}$ is a solution of the problem \eqref{AP2Eq}, which satisfies $0 \not\in \mbox{Re}\sigma(\mathbb{A}+DF_0(w^{*}))$, then, for each $0<\varepsilon \leqslant \varepsilon_0$, with $\varepsilon_0$ suficiently small, there exists an unique solution $w^{*}_{\varepsilon}$ of the problem \eqref{AP1Eq} in a neighborhood of $w$.
\end{remark}

Therefore we conclude the continuity of the set of equilibria $\{\mathcal{E}_{\varepsilon}: \varepsilon \in [0,\varepsilon_0]\}$ at $\varepsilon =0$.

\begin{remark}
\label{remarkhyperbolic}
Now that we have obtained an unique solution $w^{*}_{\varepsilon}$ for \eqref{AP1Eq} in a small neighborhood of the hyperbolic solution $w^{*}$ for \eqref{AP2Eq}, we can consider the linearization $\mathbb{A}+DF_{0}(w^{*}_{\varepsilon})$ and from the convergence of $w^{*}_{\varepsilon}$ to $w^{*}$ in $X$ it is easy to obtain that $(\mathbb{A}+DF_{\varepsilon}(w^{*}_{\varepsilon}))^{-1}w_{\varepsilon}$ converges to $(\mathbb{A}+DF_{0}(w^{*}))^{-1}w$ in $X$, whenever $w_{\varepsilon}\to w$ in $\mathcal{H}$, as $\varepsilon\to0$. Consequently, the hyperbolicity of $w^{*}$ implies the hyperbolicity of $w^{*}_{\varepsilon}$, for suitably small $\varepsilon$.
\end{remark}

\begin{theorem}
\label{theofiniteequilibria}
If all solutions $w^{*}$ of \eqref{AP2Eq} satisfy $0\notin \mbox{Re}\sigma(\mathbb{A}+DF_{0}(w^{*}))$, then \eqref{AP2Eq} has a finite number $k$ of solutions, $w^{*}_{1},...,w^{*}_{k}$, and there exists $\varepsilon_{0}>0$ such that, for each $\varepsilon \in (0, \varepsilon_{0}]$, the equation \eqref{AP1Eq} has exactly $k$ solutions, $w^{*}_{\varepsilon,1},...,w^{*}_{\varepsilon,k}$. Moreover, for all $i=1,...,k$,
\[
w^{*}_{\varepsilon,i}\to w^{*}_{i}\quad \mbox{in}\quad X, \quad \mbox{as $\varepsilon\to0$}.
\]
\end{theorem}

\begin{proof}
The proof follows of Theorems \ref{theoremisolated} and \ref{Teoemsas}. 
\end{proof}


\section{Lower semicontinuity of attractors}\label{Sec:Lower_At}

Next we show that the local unstable manifolds of $w^{*}_{\varepsilon,i}$ fixed, are continuous in $X$ as $\varepsilon\to0$. This
fact and the continuity of the set of equilibria enable us to prove the lower semicontinuity of the attractors at  $\varepsilon=0$. For this we will use the convergence results of the previous sections and the convergence of the linearized semigroups proved next.

The main aim of this section is the proof of existence unstable local manifolds as a graph of a Lipschitz function, its convergence and exponential attraction. Let us consider $w^{*}_{\varepsilon,i}$ be an equilibrium solution for \eqref{AP1}, thus $\mathbb{A} w^{*}_{\varepsilon,i}+F_\varepsilon(w^{*}_{\varepsilon,i})=0$. To deal with a neighborhood of the equilibrium solution $w^{*}_{\varepsilon,i}$, we rewrite the problem \eqref{AP1} as
\begin{equation}
\label{pusogalerkin}
\begin{cases}
\dfrac{d{\bf w}^\varepsilon}{dt}={\bf A}_\varepsilon {\bf w}^\varepsilon+F_\varepsilon({\bf w}^\varepsilon+w^{*}_{\varepsilon,i})-F_\varepsilon(w^{*}_{\varepsilon,i})-DF_\varepsilon(w^{*}_{\varepsilon,i}){\bf w}^\varepsilon, \quad t>0,\\
{\bf w}^\varepsilon(0)=w_0-w^{*}_{\varepsilon,i},
\end{cases}
\end{equation}
where ${\bf w}^\varepsilon=w^\varepsilon-w^{*}_{\varepsilon,i}$ and ${\bf A}_\varepsilon=\mathbb{A}+DF_\varepsilon(w^{*}_{\varepsilon,i})$. With this, one can look for the previous sections with the unbounded linear operator ${\bf A}_\varepsilon$ instead of the unbounded linear operator $\mathbb{A}$.

Let $\varsigma$ be a smooth, closed, simple, rectifiable curve in $\{z\in \mathbb{C}: {\textrm Re}z>0\}$, oriented counterclockwise and such that the bounded connected component of $\mathbb{C}\backslash \{\varsigma\}$; here, $\{\varsigma\}$ denotes the trace of $\varsigma$, contains $\{z\in \sigma({\bf A}_0):{\textrm Re z}>0\}$. Let $\{\varsigma\}\subset \rho({\bf A}_\varepsilon)$, for all $\varepsilon \in [0,\varepsilon_1]$, for some $\varepsilon_1>0$. We define ${\bf Q}_\varepsilon$ by
\[
{\bf Q}_\varepsilon=\frac{1}{2\pi i} \int_{\varsigma} (\lambda -{\bf A}_\varepsilon)^{-1} d\lambda,
\]
for any $\varepsilon \in [0,\varepsilon_1]$.

There exist $\beta>0$ and $C\geq 1$ such that
\[
\|e^{-{\bf A}_\varepsilon t}{\bf Q}_\varepsilon\|_{\mathcal{L}(X)}\leq Ce^{-\beta t},
\]
for any $t\geq 0$ and
\[
\|e^{-{\bf A}_\varepsilon t}(I-{\bf Q}_\varepsilon)\|_{\mathcal{L}(X)}\leq Ce^{\beta t}
\]
for any $t>0$ and $\varepsilon \in [0,\varepsilon_1]$.

Using the decomposition $X={\bf Q}_\varepsilon X\oplus(I-{\bf Q}_\varepsilon)X$ (the solution ${\bf w}^\varepsilon$ of \eqref{pusogalerkin} can be decomposed as ${\bf w}^\varepsilon={\bf Q}_\varepsilon {\bf w}^\varepsilon+(I-{\bf Q}_\varepsilon){\bf w}^\varepsilon$), we rewrite \eqref{pusogalerkin} as following
\begin{equation}\label{sist_acoplado}
\begin{cases}
\dfrac{d}{dt}({\bf Q}_\varepsilon {\bf w}^\varepsilon)={\bf A}_\varepsilon\, {\bf Q}_\varepsilon {\bf w}^\varepsilon+H_\varepsilon({\bf Q}_\varepsilon {\bf w}^\varepsilon,(I-{\bf Q}_\varepsilon){\bf w}^\varepsilon),\\ \\
\dfrac{d}{dt}[(I-{\bf Q}_\varepsilon){\bf w}^\varepsilon]={\bf A}_\varepsilon(I-{\bf Q}_\varepsilon) {\bf w}^\varepsilon+G_\varepsilon({\bf Q}_\varepsilon {\bf w}^\varepsilon,(I-{\bf Q}_\varepsilon){\bf w}^\varepsilon),
\end{cases}
\end{equation}
where
\[
\begin{split}
&H_\varepsilon({\bf Q}_\varepsilon {\bf w}^\varepsilon,(I-{\bf Q}_\varepsilon){\bf w}^\varepsilon)\\
&:={\bf Q}_\varepsilon[F({\bf Q}_\varepsilon {\bf w}^\varepsilon+(I-{\bf Q}_\varepsilon){\bf w}^\varepsilon+w^{*}_{\varepsilon,i})-F(w^{*}_{\varepsilon,i})-DF(w^{*}_{\varepsilon,i})({\bf Q}_\varepsilon {\bf w}^\varepsilon+(I-{\bf Q}_\varepsilon){\bf w}^\varepsilon)]
\end{split}
\]
and
\[
\begin{split}
&G_\varepsilon({\bf Q}_\varepsilon {\bf w}^\varepsilon,(I-{\bf Q}_\varepsilon){\bf w}^\varepsilon)\\
 &:=(I-{\bf Q}_\varepsilon)[F({\bf Q}_\varepsilon {\bf w}^\varepsilon+(I-{\bf Q}_\varepsilon){\bf w}^\varepsilon+w^{*}_{\varepsilon,i})-F(w^{*}_{\varepsilon,i})-DF(w^{*}_{\varepsilon,i})({\bf Q}_\varepsilon {\bf w}^\varepsilon+(I-{\bf Q}_\varepsilon){\bf w}^\varepsilon)].
\end{split}
\]

The maps $H_\varepsilon$ and $G_\varepsilon$ are continuously differentiable with $H_\varepsilon(0,0)=G_\varepsilon(0,0)=0$ and $DH_\varepsilon(0,0)=DG_\varepsilon(0,0)=0$. For simplicity of notation, we write $\omega^\varepsilon={\bf Q}_\varepsilon {\bf w}^\varepsilon$ and $\vartheta^\varepsilon=(I-{\bf Q}_\varepsilon){\bf w}^\varepsilon$.  Hence, given  $\rho>0$, there exist $\varepsilon_1>0$ and $r>0$ such that if $\|\omega^\varepsilon\|_{{\bf Q}_\varepsilon X}+\|\vartheta^\varepsilon\|_{(I-{\bf Q}_\varepsilon)X}<r$ and $\varepsilon \in [0,\varepsilon_1]$, then
\[
\|H_\varepsilon(\omega^\varepsilon,\vartheta^\varepsilon)\|_{{\bf Q}_\varepsilon X}\leq\rho\  \  \hbox{and}\  \  \|G_\varepsilon(\omega^\varepsilon,\vartheta^\varepsilon)\|_{(I-{\bf Q}_\varepsilon)X}\leq\rho,
\]
\[
\|H_\varepsilon(\omega^\varepsilon,\vartheta^\varepsilon)-H_\varepsilon(\bar\omega^\varepsilon,\bar\vartheta^\varepsilon)\|_{{\bf Q}_\varepsilon X}\leq\rho(\|\omega^\varepsilon-\bar\omega^\varepsilon\|_{{\bf Q}_\varepsilon X}+\|\vartheta^\varepsilon-\bar\vartheta^\varepsilon\|_{(I-{\bf Q}_\varepsilon)X})
\]
and
\[
\|G_\varepsilon(\omega^\varepsilon,\vartheta^\varepsilon)-G_\varepsilon(\bar\omega^\varepsilon,\bar\vartheta^\varepsilon)\|_{(I-{\bf Q}_\varepsilon)X}\leq\rho(\|\omega^\varepsilon-\bar\omega^\varepsilon\|_{{\bf Q}_\varepsilon X}+\|\vartheta^\varepsilon-\bar\vartheta^\varepsilon\|_{(I-{\bf Q}_\varepsilon)X}).
\]

Considering the coupled system \eqref{sist_acoplado}, we can show an unstable manifold theorem using similar arguments used in the results in Henry \cite[Chapter 6]{He}. For this, we consider the following theorem.

\begin{theorem}
\label{cont-variedades}
There exists a map $s_{*}^\varepsilon:{\bf Q}_\varepsilon X\to(I-{\bf Q}_\varepsilon)X$ such that the unstable manifold of $w^{*}_{\varepsilon,i}$ is given by
\[
W^u(w^{*}_{\varepsilon,i})=\{(\omega,\vartheta)\in X;\ \vartheta=s_{*}^\varepsilon(\omega),\ \omega\in{\bf Q}_\varepsilon X\}.
\]
The map $s_{*}^\varepsilon$ satisfies
\[
|\!|\!|s_{*}^\varepsilon|\!|\!|:=\sup_{\omega\in{\bf Q}_\varepsilon X}\|s_{*}^\varepsilon(\omega)\|_{X}\leq C_{Lip},\ \ \
\|s_{*}^\varepsilon(\omega)-s_{*}^\varepsilon(\widetilde{\omega})\|_{X}\leq \bar C_{Lip}\|\omega-\widetilde{\omega}\|_{{\bf Q}_\varepsilon X},
\]
where $C_{Lip}>0$ is constant independent of $\varepsilon$, and 
\[
|\!|\!|s^\varepsilon_{*}-s^0_{*}|\!|\!|\to0, \quad \mbox{as $\varepsilon\to0$}.
\]
Furthermore, there exist $\rho_1>0$ and $k>0$, independents of $\varepsilon$, and $t_0>0$ such that, for any solution $(\omega^\varepsilon(t),\vartheta^\varepsilon(t))\in X$ ($t \in [t_0,\infty)$) of \eqref{sist_acoplado}, we have
\[
\|\vartheta^\varepsilon(t)-s^\varepsilon_{*}(\omega^\varepsilon(t))\|_{X}\leq ke^{-\rho_1(t-t_0)}\|\vartheta^\varepsilon(t_0)-s^\varepsilon_{*}(\omega^\varepsilon(t_0))\|_{X},\quad \forall t\geq t_0.
\]
\end{theorem}

\begin{proof}
Thanks to the results of previous sections, the proof follows using arguments already known in the literature, see e.g. Henry \cite[Chapter 6]{He}.
\end{proof}

\begin{theorem}
The family of global attractors $\{\mathcal{A}_\varepsilon: \varepsilon \in [0,\varepsilon_0] \}$ is lower semicontinuous at $\varepsilon=0$; that is,
\[
\operatorname{dist}_{X}(\mathcal{A}_0,\mathcal{A}_\varepsilon) \to 0, \quad \mbox{as} \quad \varepsilon \to 0,
\]
where
\[
\operatorname{dist}_{X}(\mathcal{A}_0,\mathcal{A}_\varepsilon):=\sup_{w_0 \in\mathcal{A}_0}\inf_{w_\varepsilon \in\mathcal{A}_\varepsilon}\{\| w_{\varepsilon}-w_0 \|_{X}\}.
\]
\end{theorem}

\begin{proof}
Thanks to the results of previous sections, the proof follows using arguments already known in the literature, see e.g. \cite[Chapter 3, Section 3.3]{CLR}. 
Let $w\in\mathcal{A}_0$. Since $\{S_{0}(t): t \geqslant 0\}$ is a gradient system, we have that
\[
\mathcal{A}_0 = \bigcup_{w^*\in\mathcal{E}_0}W^u(w^*)
\]
and then $w\in W^u(w^*)$, for some $w^*\in\mathcal{E}_0$. Let $\tau\in\mathbb{R}$ and $\varphi\in W^u_{loc}(w^*)$  be such that $S_{0}(\tau)\varphi=w$. Let $w_\varepsilon^*$ be such that $w_\varepsilon^*\to w^*$ as $\varepsilon\to0$. From the convergence of unstable manifolds there is a sequence $\{\varphi_\varepsilon\}_{\varepsilon\in(0,\varepsilon_0]}$, $\varphi_\varepsilon\in W^u_{loc}(w^*_\varepsilon)$ with $w_{\varepsilon}^{*}\in\mathcal{E}_\varepsilon$, such that $\varphi_\varepsilon\to\varphi$ as $\varepsilon\to0$. Finally, from Proposition \ref{proposition4.5}, we obtain $S_{\varepsilon}(\tau)\varphi_\varepsilon\to S_{0}(\tau)\varphi=w$ as $\varepsilon\to0$. To conclude, we observe that if $w_\varepsilon=S_{\varepsilon}(\tau)\varphi_{\varepsilon}$, then $w_\varepsilon\in\mathcal{A}_\varepsilon$, since 
\[
\varphi_\varepsilon\in \bigcup_{w_\varepsilon^*\in\mathcal{E}_\varepsilon}W^u(w^*_\varepsilon)=\mathcal{A}_\varepsilon
\] 
and $\mathcal{A}_\varepsilon$ is invariant.
\end{proof}

\begin{corollary}
The family of global attractors $\{\mathcal{A}_\varepsilon:\varepsilon\in(0,\varepsilon_0]\}$ is continuous at $\varepsilon=0$.
\end{corollary}

\end{document}